\documentclass[11pt,reqno]{amsart} 

\makeatletter
\ifx\SetFigFont\undefined\gdef\SetFigFont#1#2#3#4#5{\reset@font\fontsize{#1}{#2pt}\fontfamily{#3}\fontseries{#4}\fontshape{#5}\selectfont}\fi

\usepackage{amssymb,amsfonts,amsmath,epsfig}
\numberwithin{equation}{section}

\font\script=rsfs10 at 11pt
\def\H{{\mbox{\script H}\,\,}}
\def\LL{{\mbox{\script L}\,\,}}
\def\II{{\mbox{\script I}\,\,}}
\def\MM{{\mbox{\script M}\,\,}}
\def\R{\mathbb R}
\def\B{\mathcal B}
\def\D{\mathcal D}
\def\Q{\mathcal Q}
\def\N{\mathbb N}
\def\T{\mathcal T}
\def\eps{\varepsilon}

\def\angle#1#2#3{#1\widehat{#2}#3}
\def\uad{u_{\rm adj}}
\def\upqd{u^\delta_{pq}}

\def\clos{\mathrm{clos}\,}
\def\ee{\mathrm{e}}
\def\bz{\bar{z}}

\def\XXint#1#2#3{{\setbox0=\hbox{$#1{#2#3}{\int}$} \vcenter{\vspace{-1pt}\hbox{$#2#3$}}\kern-.5\wd0}}
\def\Xint#1{\mathchoice {\XXint\displaystyle\textstyle{#1}}{\XXint\textstyle\scriptstyle{#1}}{\XXint\scriptstyle\scriptscriptstyle{#1}}{\XXint\scriptscriptstyle\scriptscriptstyle{#1}}\!\int}
\def\intmed{\Xint{-}}

\def\step#1#2{\par\medskip \noindent{\underline{\it Step~#1.}}\emph{ #2}\par}

\newtheorem{theorem}{Theorem}[section]
\newtheorem{lemma}[theorem]{Lemma}

\newtheorem{proposition}[theorem]{Proposition}
\newtheorem{remark}[theorem]{Remark}

\newtheorem{definition}[theorem]{Definition}

\begin{document}

\title[Approximation of orientation-preserving maps]{Smooth approximation of bi-Lipschitz orientation-preserving homeomorphisms}

\author{Sara Daneri}\address{Dipartimento di Matematica, via Ferrata 1, 27100 Pavia (Italy)}\email{sara.daneri@unipv.it}
\author{Aldo Pratelli}\address{Dipartimento di Matematica, via Ferrata 1, 27100 Pavia (Italy)}\email{aldo.pratelli@unipv.it}

\begin{abstract}
We show that a planar bi-Lipschitz orientation-preserving homeomorphism can be approximated in the $W^{1,p}$ norm, together with its inverse, with an orientation-preserving homeomorphism which is piecewise affine or smooth.
\end{abstract}

\maketitle

\section{Introduction}
In this paper we deal with approximations of bi-Lipschitz orientation-preserving homeomorphisms $u:\Omega\subseteq\R^2\longrightarrow\Delta\subseteq\R^2$, where $\Omega$ and $\Delta=u(\Omega)$ are two open bounded subsets of $\R^2$.
 In particular, we show that both $u$ and its inverse can be approximated in the $W^{1,p}$-norm ($p\in[1,+\infty)$) by piecewise affine or smooth orientation-preserving homeomorphisms. Our main theorem is the following.

\begin{theorem}\label{main}
Let $\Omega\subseteq \R^2$ be any bounded open set, and let $u:\Omega\longrightarrow\Delta$ be a bi-Lipschitz orientation-preserving homeomorphism. Then, for any $\bar\eps>0$ and any $1\leq p<\infty$, there exists a bi-Lipschitz orientation-preserving homeomorphism $v:\Omega\longrightarrow \Delta$ such that $u= v$ on $\partial \Omega$,
\begin{equation}\label{approximation}
\|{u-v}\|_{L^{\infty}(\Omega)}+\|{u^{-1}-v^{-1}}\|_{L^{\infty}(\Delta)}+\|{Du-Dv}\|_{L^{p}(\Omega)}+\|{Du^{-1}-Dv^{-1}}\|_{L^{p}(\Delta)} \leq \bar\eps\,,
\end{equation}
and $v$ is either countably piecewise affine or smooth. More precisely, there exist two geometric constants $C_1$ and $C_2$ such that, if $u$ is $L$ bi-Lipschitz, then the countably piecewise affine approximation can be chosen to be $C_1 L^4$ bi-Lipschitz, while the smooth approximation can be chosen to be $C_2 L^{28/3}$ bi-Lipschitz.
\end{theorem}

Thanks to a result by Mora-Corral and the second author~\cite{MCPra} (see Theorem~\ref{MP} below), the problem of finding smooth approximations can be actually reduced to find countably piecewise affine ones --i.e. affine on the elements of a locally finite triangulation of $\Omega$, see Definition~\ref{def:triang}.

The fact that $v$ might not be (finitely) piecewise affine but countably piecewise affine is due to the fact that we require $u=v$ on $\partial\Omega$, so it is clearly impossible to find a (finitely) piecewise affine approximation $v$ unless the domain is a polygon and $u$ is piecewise affine on the boundary. In fact, we also prove the following result.

\begin{theorem}\label{mainaffine}
If under the assumptions of Theorem~\ref{main} one has also that $\Omega$ is polygonal and $u$ is piecewise affine on $\partial \Omega$, then there exists a (finitely) piecewise affine approximation $v:\Omega\longrightarrow \Delta$ as in Theorem~\ref{main} which is $C_1 C'(\Omega) L^4$ bi-Lipschitz. 
\end{theorem}

About the dependence of $C'(\Omega)$ in Theorem~\ref{mainaffine} on the domain $\Omega$, see Remark~\ref{depC'}.\par
The first naive idea coming to one's mind in order to construct a piecewise affine approximation of $u$ could be the following: first, to select an arbitrary locally affine triangulation of $\Omega$ with triangles of sufficiently small diameter; then, to define $v$ as the function which, on every triangle, is the affine interpolation of the values of $u$ on its vertices.
Unfortunately, if on one hand the functions defined in this way provide an approximation of $u$ in $L^{\infty}$, on the other hand they may fail to be homeomorphisms. The problem is due to the fact that, taking arbitrary nondegenerate triangles in $\Omega$ --no matter how small-- then the affine interpolation of $u$ on the vertices of the triangles can be orientation-preserving on some triangles and orientation-reversing on the others (see Figure~\ref{fig1}). This prevents the affine interpolation to be injective since an homeomorphism on a connected domain in $\R^2$ must be either orientation-preserving on every subdomain, or orientation-reversing on every subdomain. An explicit example of a function with such a bad behaviour can be found in~\cite{SerShi}.\par
\begin{figure}[htbp]
\begin{center}
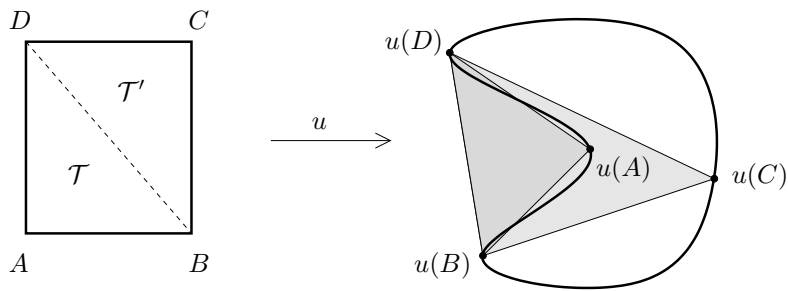
\caption{{\small The square $ABCD$ is divided in the triangles $\T$ and $\T'$. The affine interpolation $v$ of $u$ on $ABCD$ is not injective, since $v(\T)\subseteq v(\T')$ ($v(\T)$ and $v(\T')$ are shaded). Moreover, $u$ is orientation-preserving in the square while $v$ is orientation-reversing on $\T$.}}\label{fig1}
\end{center}
\end{figure}
The general problem of finding suitable approximations of homeomorphisms $u:\R^d\supseteq\Omega\longrightarrow u(\Omega)\subseteq\R^d$ with piecewise affine homeomorphisms has a long history. As far as we know, in the simplest non-trivial setting (i.e. $d=2$, approximations in the $L^{\infty}$-norm) the problem was solved by Rad\'{o}~\cite{Rado}. Due to its fundamental importance in geometric topology, the problem of finding piecewise affine homeomorphic approximations in the $L^{\infty}$-norm and dimensions $d>2$ was deeply investigated in the 50s and 60s. In particular, it was solved by Moise~\cite{Moise1} and Bing~\cite{Bing} in the case $d=3$ (see also the survey book~\cite{Moise2}), while for contractible spaces of dimension $d\geq5$ the result follows from theorems of Connell~\cite{Conn}, Bing~\cite{Bing2}, Kirby~\cite{Kirby} and Kirby, Siebenmann and Wall~\cite{Kirbetal} (for a proof see, e.g., Rushing~\cite{Rush} or Luukkainen~\cite{Lukk}). Finally, twenty years later, while studying the class of quasi-conformal varietes, Donaldson and Sullivan~\cite{DonSull} proved that the result is false in dimension 4.\par

Let us now consider the case of Sobolev homeomorphisms $u\in W^{1,p}$ for some $p\in[1,+\infty]$. As pointed out by Ball (see~\cite{Ball,Ball2}, see also Evans~\cite{Evans}), the problem of proving the existence of piecewise affine approximations of Sobolev homeomorphisms arises naturally when one wants to approximate with finite elements the solutions of minimization problems in nonlinear elasticity (e.g. the minima of neohookean functionals, see also~\cite{Ball1}, \cite{BauPhiOw}, \cite{CDeL}, \cite{SivSpec}). In that context, the function $u$ represents the physical deformation of a material with no interpenetration of matter (in particular, $d=2$ as in the present paper, or $d=3$). 

The results available in the literature provide, under increasingly weaker hypotheses on the derivatives of $u$, piecewise affine or smooth approximations of $u$ and its derivatives.
The first results were obtained by Mora-Corral~\cite{MC} (for planar bi-Sobolev mappings that are smooth outside a finite set) and by Bellido and Mora-Corral~\cite{BMC}, in which they prove that if $u\in C^{0,\alpha}$ for some $\alpha\in (0,1]$, then one can find piecewise affine approximations $v$ in $C^{0,\beta}$, where $\beta\in(0,\alpha)$ depends only on $\alpha$.\par
Recently, Iwaniec, Kovalev and Onninen~\cite{IwKovOnn} almost completely solved the approximation problem of planar Sobolev homeomorphisms, proving that whenever $u$ belongs to $W^{1,p}$ for some $1<p<+\infty$, then it can be approximated by smooth diffeomorphisms $v$ in the $W^{1,p}$-norm (improving the previous result for homeomorphisms in $W^{1,2}$ found in~\cite{IKO1}).\par

In all the above-mentioned results, approximations are intended in the $W^{1,p}$ sense, that is, in place of~(\ref{approximation}) one obtains estimates of the form
\begin{equation}\label{oldapp}
\|{u-v}\|_{L^{p}(\Omega)}+\|{Du-Dv}\|_{L^{p}(\Omega)} \leq \bar\eps\,,
\end{equation}
whitout any information on the inverses $u^{-1}$ and $v^{-1}$. Unfortunately, this is still not enough for the applications: in fact, the functionals of nonlinear elasticity usually depend on functions of the Jacobian of $u$ which explode when ${\rm det} (Du)\to 0$ (see for instance~\cite[pag.~3]{Ball2}). The physical meaning of choosing such functionals is that too high compressions or strecthings require high energy. As a consequence of this, two invertible Sobolev functions $u$ and $v$ which are close in the sense of~(\ref{oldapp}) may have very different energies. Instead, if $u$ and $v$ are close in the sense of~(\ref{approximation}), then their energies are also close.\par

These considerations suggest to work in the space of bi-Sobolev homeomorphisms, that is, homeomorphisms $u$ such that both $u$ and $u^{-1}$ belong to $W^{1,p}$. This was also already suggested in the paper by Iwaniec, Kovalev and Onninen (\cite[Question~4.2]{IwKovOnn}). We only mention here that the study of bi-Sobolev homeomorphisms is very active, also in connection with maps of finite distorsion (see for instance~\cite{CsHM,Hencl1,Hencl2,HenclMaly,HenclKosk,HenclKoskMaly,HenclMPS}).\par


The present paper is the first one to take care also of the distance of the inverse maps, leading to a partial result towards the solution of the general problem (hence, we also give a partial positive answer to Question~4.2 of~\cite{IwKovOnn}): in fact, we are able to deal with homeomorphisms which are bi-Sobolev for $p=+\infty$. The techniques adopted in~\cite{BMC} and~\cite{IwKovOnn} are completely different with respect to the ones which will be used throughout this paper. While the proof in~\cite{BMC} is based on a refinement of the supremum norm approximation of Moise~\cite{Moise1} (which, as pointed out by the authors themselves, cannot be extended to deal with the Sobolev case) and the approach of~\cite{IwKovOnn} makes use of the identification $\R^2\simeq \mathbb C$ and involves coordinate-wise p-harmonic functions, our proof is constructive and based on an explicit subdivision of the domain of $u$ depending on the Lebesgue points of $Du$.\par
We conclude giving a bound on the values of the costants $C_1$, $C_2$ and $C_3$ appearing in Theorems~\ref{main}, \ref{mainaffine} and~\ref{teo:bilsquare} (while the constant $C'(\Omega)$ depends on the set $\Omega$, see Remark~\ref{depC'}),
\begin{align*}
C_1= 72^4 C_3\,, && C_2 = 70C_1^{7/3}\,, && C_3 = 636000 \,.
\end{align*}

\section{Scheme of the proof and plan of the paper\label{sect:scheme}}

Our proof is constructive, thus long, but it relies essentially on three known facts: the Lebesgue differentiation Theorem for $L^1$-maps in $\R^d$, the Jordan curve theorem and a planar bi-Lipschitz extension theorem for homeomorphic images of squares proved in~\cite{DanPra}. As already mentioned in the introduction, all our effort will be to get a piecewise affine approximation of $u$, since then the smooth extension readily follows by the following recent result from~\cite{MCPra}.
\begin{theorem}\label{MP}
Let $v:\Omega\longrightarrow\R^2$ be a (countably) piecewise affine homeomorphism, bi-Lipschitz with constant $L$. Then there exists a smooth diffeomorphism $\hat v:\Omega\longrightarrow v(\Omega)$ such that $\hat v \equiv v$ on $\partial \Omega$, $\hat v$ is bi-Lipschitz with constant at most $70L^{7/3}$, and
\[
\|\hat v - v\|_{L^{\infty}(\Omega)} + \| D\hat v - Dv\|_{L^p(\Omega)} + \|\hat v^{-1} - v^{-1}\|_{L^{\infty}(v(\Omega))} + \| D\hat v^{-1} - Dv^{-1}\|_{L^p(v(\Omega))} \leq \eps\,.
\]
\end{theorem}

A rough idea and scheme of our construction is as follows. 

\paragraph{\textbf{Approximation of $\mathbf{u}$ on Lebesgue squares}}
The first idea is to use the fact that, in a sufficiently small neighborhood of each Lebesgue point $z$ for the differential $Du$, the map $u$ is arbitrarily close, both in $W^{1,p}$ and in $L^{\infty}$, to an affine $L$ bi-Lipschitz map (given by its linearization around the point $z$). The $W^{1,p}$ estimate is simply a restatement of the definition of Lebesgue point of $Du$, while the $L^{\infty}$ estimate is proven in Lemma~\ref{lem:linfty1}.
Indeed we prove that, given a square $\D\subseteq\Omega$ (e.g. a neighborhood of $z$), the more $Du$ is close in $L^{p}(\D)$ to an $L$ bi-Lipschitz matrix $M$ (given e.g. by $Du(z)$), the more $u$ is close in $L^{\infty}(\D)$ to an $L$ bi-Lipschitz affine map $u_M$ with $Du_M=M$. Moreover, since $u$ is bi-Lipschitz, we have that also the inverse of $u$ is close both in $W^{1,p}(u(\D))$ and in $L^{\infty}(u(\D))$ to the inverse of $u_M$.

The main implication of these estimates towards the construction of a piecewise affine bi-Lipschitz map approximating $u$ is the following.
Let us take a square $\D\subseteq\Omega$ as above and let us consider the piecewise affine function $v$ which coincides with $u$ on the vertices of $\D$ and is affine on each of the two triangles obtained dividing $\D$ with a diagonal. If $\|Du-M\|_{L^p(\D)}$ is sufficiently small, then the $L^{\infty}$ estimate implies that $u(\partial\D)$ is uniformly relatively close to the parallelogram of side lengths at least $\mathrm{side}(\D)/ L$ given by $u_M(\partial\D)$.
Hence, since $v=u$ on the vertices of the square and is affine on each side of $\partial\D$, the same uniform estimate holds also for $v$. In particular, the map $v$ is orientation preserving, injective, and approximates $u$ and its inverse as desired. 

Finally, thanks to the fact that the Lebesgue points of $Du$ have full measure in $\Omega$, we fix two orthonormal vectors $\ee_1$, $\ee_2\in\R^2$ and, $\forall\,\eps>0$, we find a set $\Omega_\eps\subset\subset\Omega$ with $\LL(\Omega\setminus\Omega_\eps)\leq\eps$ which is made by a uniform ``tiling'' of squares with sides parallel to $\ee_1$, $\ee_2$ with the following property. On each square $\D$ of the tiling, $Du$ is sufficiently close to an $L$ bi-Lipschitz matrix $M$ (in particular, $M$ will be equal to $Du(z)$ for some Lebesgue point $z\in\D$). Then, by the previous remarks, one can show that the piecewise affine function $v$ obtained interpolating between the values of $u$ on the vertices of the squares is injective and satisfies~\eqref{approximation} on $\Omega_\eps$. Moreover, $v$ is $L+\eps$ bi-Lipschitz.
The squares of the tiling covering $\Omega_\eps$ will be called \emph{Lebesgue squares}, and the set $\Omega_\eps$ \emph{right polygon}, due to its shape --see Figure~\ref{Fig:rightpolygon}.
\begin{figure}[htbp]
\begin{center}
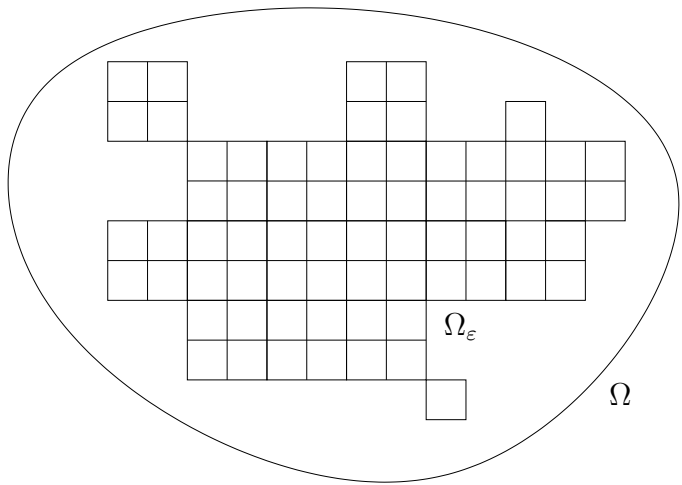
\caption{An open set $\Omega$ and a right polygon $\Omega_\eps\subset\subset\Omega$.}\label{Fig:rightpolygon}
\end{center}
\end{figure}

Thus, the first idea of the proof leads to define a piecewise affine approximation of $u$ on a set whose Lebesgue measure is as close as we want to $\LL(\Omega)$. In order to complete the construction, we have to define $v$ in the interior of the set $\Omega\setminus\Omega_\eps$.

\paragraph{\textbf{Countably piecewise affine bi-Lipschitz extension}}

The second idea of our proof is to reduce to the following model case: $\Omega\setminus\Omega_\eps$ is a square of Lebesgue measure at most $\eps$ and $u_{|_{\partial (\Omega\setminus\Omega_\eps)}}$ is a piecewise affine function. In particular, by the previous construction, $v=u$ on $\partial\Omega_\eps$. In this case, an approximating function $v$ is provided by the following bi-Lipschitz extension theorem, proved in~\cite{DanPra}.

\begin{theorem}[\cite{DanPra}]\label{teo:bilsquare}
 There exists a geometric constant $C_3$ such that every $L$ bi-Lipschitz piecewise affine map $u:\partial\D(0,1)\longrightarrow \R^2$ defined on the boundary of the unit square admits a $C_3L^4$ bi-Lipschitz piecewise affine extension $v:\D(0,1)\longrightarrow\Gamma$, where $\Gamma$ is the bounded closed set such that $\partial\Gamma=u(\partial\D(0,1))$.
\end{theorem}

In particular, it is shown in~\cite{DanPra} that one can take $C_3=636000$. We recall that an analogous result had already been proved by Tukia in~\cite{Tukia}. However, explicit estimates of the Lipschitz constant of the extension $v$ were not provided.

It is then easy to verify that, provided $\eps$ is chosen sufficiently small at the beginning, such an extension of $u_{|_{\partial (\Omega\setminus\Omega_\eps)}}$ together with the already defined piecewise affine interpolation of $u$ on the Lebesgue squares, satisfies the assumptions of Theorem~\ref{main}.
Indeed, by definition, $v$ is injective on the whole $\Omega$. Moreover, we know by the previous construction that it satisfies~\eqref{approximation} on $\Omega_\eps$. And on the other hand, on $\Omega\setminus\Omega_\eps$, $|Du|$ and $|Dv|$ are bounded by the two Lipschitz constants $L$ and $CL^4$ (together with their inverses) on a set of small area and then the $W^{1,p}$ estimates in~\eqref{approximation} follow. Finally, since $\Omega\setminus\Omega_\eps$ and $u(\Omega\setminus\Omega_\eps)$ have small Lebesgue measure, $v$ and $v^{-1}$ are also close to $u$ and $u^{-1}$ in $L^{\infty}$. 

\begin{figure}[htbp]
\begin{center}
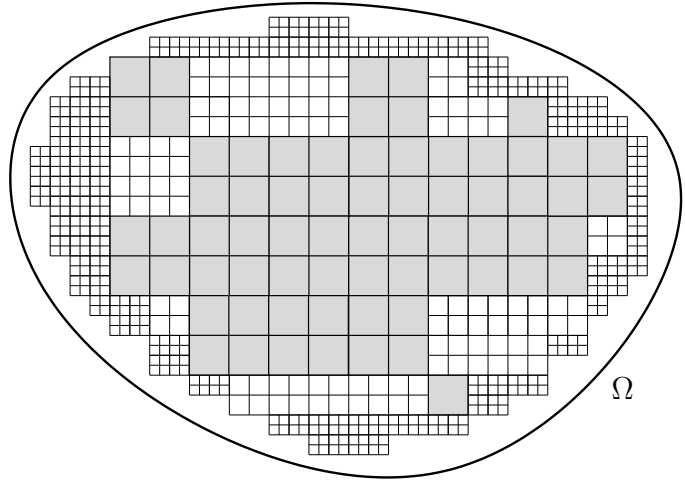
\caption{(A part of) the countable tiling of $\Omega\setminus \Omega_\eps$ --the shaded region is $\Omega_\eps$.}\label{Fig:poltil}
\end{center}
\end{figure}

In order to reduce to this model case, we perform the following steps:

\textbf{1. }We cover $\Omega\setminus\Omega_\eps$ with a countable (locally finite in $\Omega$) ``tiling'' of small squares whose sides are parallel to $\ee_1$ and $\ee_2$ (see Figure~\ref{Fig:poltil}).

\textbf{2. }On the $1$-dimensional grid $\Q$ given by the boundaries of the squares of the tiling we define the piecewise affine approximation $v$ in such a way that $v(\Q)\subseteq\Delta$ and $v$ is $72L$ bi-Lipschitz.

\textbf{3. } We ``fill'' the squares of the tiling extending $v_{|_Q}$ by means of Theorem~\ref{teo:bilsquare}, thus getting a globally $C_3(72L)^4$ bi-Lipschitz function on $\Omega\setminus\Omega_\eps$.

The fact that the Lipschitz constant of $v$ on $\Omega_\eps$ depends only on the Lipschitz constant of $u$ will tell us that, as in the model case, the $W^{1,p}$ and $L^{\infty}$ norms of $u-v$ and $u^{-1}-v^{-1}$ can be made as small as we want --provided we choose $\eps$ sufficiently small at the beginning. Thus we end the proof of Theorem~\ref{main}.

Let us also give a very rough idea of how the proofs of Steps~{\bf 1}, {\bf 2} and~{\bf 3} works. While Step~{\bf 1} is a simple geometric construction, Step~{\bf 2} essentially consists in approximating $u$ on the grid $\Q$ with a piecewise affine function. This will be possible thanks to Lemma~\ref{lem:curves}, which tells that it is possible to approximate $u$ on the segments, and to Lemma~\ref{lem:cross}, which takes care of the ``crosses''. Finally, Proposition~\ref{prop:lipext} in Section~\ref{sect:2} concludes the argument of Steps~{\bf 2} and~{\bf 3}. The essential idea there is that, since on the ``non-Lebesgue squares'' the behaviour of $u$ is wilder, one cannot simply take $v$ equal to the affine interpolation of $u$ on the vertices. Indeed, as already pointed out in the introduction, this could easily give a non-injective function. However, since the total area of the non-Lebesgue squares is small, \emph{any} bi-Lipschitz function which preserves the boundary values approximates $u$ as in~(\ref{approximation}). This is why we use the extension of $v_{|Q}$ given by Theorem~\ref{teo:bilsquare} in Step~{\bf 3}.

\subsection*{Plan of the paper}

Section~\ref{sect:0} contains the main notation and some preliminary definitions.

Section~\ref{sect:1} is devoted to the construction of $v$ on a large right polygon made of ``Lebesgue squares''.

In Section~\ref{sect:2} we complete the proof of Theorem~\ref{main} using the reduction argument outlined above and the bi-Lipschitz extension Theorem~\ref{teo:bilsquare}.

Finally, in Section~\ref{sect:3}, we adapt the proof of Theorem~\ref{main} to provide a (finitely) piecewise affine approximation of $u$ whenever $\Omega$ is polygonal and $u_{|_{\partial\Omega}}$ is piecewise affine, thus proving Theorem~\ref{mainaffine}.

\section{Preliminary Notation and Definitions\label{sect:0}}

In this section we give some preliminary definitions and fix some useful notation which will be used throughout the paper.

First we recall the definition of orientation-preserving (resp. reversing) homeomorphism.

\begin{definition}[Orientation-preserving (reversing) homeomorphism]
We say that an homeomorphism $u:\Omega\longrightarrow u(\Omega)\subseteq\R^2$ is orientation-preserving (reversing) if whenever a simple closed curve $[0,1]\ni t\mapsto \gamma(t)\in \Omega$ is parameterized clockwise, then $[0,1]\ni t\mapsto u(\gamma(t))\in u(\Omega)$ is parameterized clockwise (resp. anti-clockwise).
\end{definition}
It is well known that if $\Omega$ is connected, then any homeomorphism $u:\Omega\longrightarrow u(\Omega)\subseteq\R^2$ is either orientation-preserving or orientation-reversing.

Next, we define the class of functions in which we look for approximations of $u$. To this aim we recall the definitions of \emph{(finite) triangulation} of a polygon and of \emph{locally finite triangulation} of an open set $\Omega\subseteq\R^2$.

\begin{definition}[(Finite) triangulation]\label{def:triang}
A \emph{(finite) triangulation} of a polygon $\Omega'\subseteq\R^2$ is a finite collection of closed triangles $\{T_i\}_{i=1}^N$ whose union is equal to $ \mathrm{clos}\,\Omega'$ and, for all $i\neq j$,
\begin{equation}\label{eq:titj}
 T_i\cap T_j\quad\text{is either empty, or a common vertex, or a common side of $T_i$ and $T_j$}\, .
\end{equation}
\end{definition}

\begin{definition}[Locally finite triangulation]
Let $\Omega\subseteq\R^2$ be a bounded open set. A \emph{locally finite triangulation} of $\Omega$ is a locally finite collection of closed triangles $\{T_{i}\}_{i\in\N}$ such that $\Omega\subseteq \bigcup_{i\in\N} T_i\subseteq\mathrm{clos}\, \Omega$ and~\eqref{eq:titj} holds.
\end{definition}

We notice that, unless $\Omega$ is a polygon, the number of elements of a triangulation cannot be finite.

\begin{definition}[Piecewise affine and countably piecewise affine function]
A function $v:\Omega\longrightarrow\R^2$ is \emph{countably piecewise affine} if $v_{|_T}$ is affine on every triangle $T$ of a suitable locally finite triangulation of $\Omega$. If $\Omega$ is a polygon and the triangulation is finite, then we say that $v$ is \emph{(finitely) piecewise affine}.
\end{definition}

In order to build the triangulation on which the function $v$ of Theorem~\ref{main} is countably piecewise affine, we will use, on a subset of $\Omega$ of Lebesgue measure as close as we want to $\LL(\Omega)$, uniform triangulations into right triangles. The union of such triangles will be called a right polygon, according to the following definition. From now on, $\ee_1,\ee_2$ will be two fixed orthonormal vectors in $\R^2$.

\begin{definition}[Right polygon and $r$-piecewise affine function]\label{lastdef}
An open bounded set $\Omega'\subset\R^2$ is called a \emph{right polygon of side-length $r$} (or simply an \emph{$r$-polygon}) if it is a finite union of closed polygons whose sides are all parallel to $\ee_1$, $\ee_2$, and have lengths which are integer multiples of $r>0$. Let now $\Omega'$ be an $r$-polygon, and consider a bi-Lipschitz function $u:\Omega'\to\R^2$. For every side $\Gamma\subseteq \partial\Omega'$, write $\Gamma=\cup_{i=1}^N \Gamma_i$ where the $\Gamma_i$'s are essentially disjoint segments of length $r$. We say that $u$ is \emph{$r$-piecewise affine on $\partial\Omega'$} if for any such segment $\Gamma$ and for any $i$, the function $u$ is affine on $\Gamma_i$.
\end{definition}

Points in $\Omega$ will be denoted by $z\in\R^2$ or by $(x,y)\in\R\times\R$, with $z=x\ee_1+y\ee_2$.
We denote with $\B(z,r)$ the ball of center $z$ and radius $r$ and with $\D(z,r)$ the square of center $z$, side length $r$ and sides parallel to $\ee_1$, $\ee_2$.
Moreover, the generic square of a collection of squares $\{\D(z_\alpha,r_\alpha)\}_{\alpha\in\N}$ will be also sometimes denoted simply by $\D_\alpha$. Instead of working directly with triangulations, it will be convenient, in order to apply our method, to subdivide $\Omega$ into a countable and locally finite family of squares called tiling.

\begin{definition}[Tiling]
Given an open, bounded set $\Omega$, a \emph{tiling of $\Omega$} is a locally finite (in $\Omega$) collection of closed squares $\{\D_{\alpha}(z_\alpha, r_\alpha)\}_{\alpha\in\N}$ whose union is contained between $\Omega$ and $\mathrm{clos}\,\Omega$ and such that, $\forall\alpha\neq\beta\in\N$, $\D_\alpha \cap \D_\beta$ is either empty, or a common vertex of $\D_\alpha$ and $\D_\beta$, or a side of one of the two. Two squares of a tiling are said to be \emph{adjacent} if their intersection is nonempty.
\end{definition}

Notice that a tiling of $\Omega$ can be either finite or countable and in particular it is surely countable if $\Omega$ is not a right polygon.\par

It will be often useful to regard a given tiling of $\Omega$ as the union of the finite tiling corresponding to a right polygon $\Omega'\subset\subset \Omega$ and a countable tiling of $\Omega\setminus\Omega'$, locally finite in $\Omega$. Since these kinds of ``sub-tilings'' will be frequently used in the paper, we define them separately.

\begin{definition}[r-Tiling of a right polygon and tiling of $(\Omega,\Omega')$]\label{def:tiling}
Given an $r$-polygon $\Omega'$, the \emph{$r$-tiling of $\Omega'$} is the (unique) finite collection of closed squares $\{\D(z_\alpha,r)\}_{\alpha\in \II(r)}$ whose union is equal to $\mathrm{clos}\,\Omega'$ and, $\forall\,\alpha\neq\beta\in\II(r)$, $\D_\alpha\cap\D_\beta$ is either empty, or a common vertex, or a common side of $\D_\alpha$ and $\D_\beta$. Given a bounded, open set $\Omega$ and an $r$-polygon $\Omega'\subset\subset \Omega$, a \emph{tiling of $(\Omega,\Omega')$} is a tiling of $\Omega$ whose restriction to $\Omega'$ is the $r$-tiling of $\Omega'$.
\end{definition}

The $1$-dimensional skeleton of a tiling will be called grid, according to the following definition.

\begin{definition}[Grid]\label{def:grid}
Let $\{\D_{\alpha}\}_{\alpha\in\N}$ be a tiling of $\Omega$. We call \emph{grid of the tiling} the $1$-dimensional set given by the union of the boundaries of the squares of the tiling. Each side (resp. vertex) of the squares of a tiling will be called \emph{side} (resp. \emph{vertex}) of the grid.
\end{definition}

A possible definition of a piecewise affine approximation of $u$ on a given $r$-polygon, which will be used in Section~\ref{sect:1}, is the following.

\begin{definition}[$(\Omega',r)$-interpolation of $u$]\label{def:uinterp} 
Let $\Omega'\subset\subset\Omega$ and $\{\D_\alpha\}_{\alpha\in\II(r)}$ be an $r$-right polygon and its $r$-tiling.
We call \emph{$(\Omega',r)$-interpolation of $u$} the piecewise affine function $v:\Omega'\longrightarrow v(\Omega')\subseteq\R^2$ which coincides with $u$ on the vertices of the $r$-tiling and, for each $\alpha\in \II(r)$, is affine on the two right triangles forming $\D_\alpha$ and having as common hypothenuse the north-east$/$south-west diagonal of $\D_\alpha$.
\end{definition}

We conclude this section with a table collecting the main notation used in this paper.
{\small\begin{flalign*}
\begin{array}{ll}
\Omega\subseteq\R^2&\hbox{a given open bounded set}\,,\\
u:\Omega\longrightarrow\Delta &\hbox{a given $L$ bi-Lipschitz function}\,,\\
\MM(2\times2) & \hbox{two by two real matrices}\,,\\
|M| & \sup\big\{\big|Mv\big|:\,|v|=1\big\}\,,\\
\MM(2\times2;L) & \big\{M\in\MM(2\times2):\mathrm{Det} \,M>0,\\
& \qquad\,|M|\leq L, |M^{-1}|\leq L\big\}\,,\\
\ee_1,\,\ee_2 &\hbox{two fixed positively oriented}\\
& \qquad\hbox{orthonormal vectors in $\R^2$}\,,\\
\B(z,r) & \hbox{ball with center $z$ and radius $r$}\,,
\end{array}
&&
\begin{array}{ll}
\D(z,r) & \hbox{square with center $z$, side length $r$}\\
&\qquad \hbox{and sides parallel to $\ee_1$, $\ee_2$}\,,\\
\LL & \hbox{Lebesgue measure on $\R^2$}\,,\\
\H^1 & \hbox{$1$-dimensional Hausdorff measure}\,,\\
\mathrm{int}\, A & \hbox{interior of a set $A\subseteq\R^2$}\,,\\
\clos\,A & \hbox{closure of $A$}\,,\\
\partial A & \hbox{boundary of $A\subseteq \R^2$}\,,\\
\Omega'\subset\subset\Omega & \mathrm{clos}\,{\Omega'}\subseteq\Omega\,,\\
d(A,B) & \inf\{|z-w|:\,z\in A,\,w\in B\}\,.
\end{array}
\end{flalign*}}

\section{Approximation on the ``Lebesgue squares''\label{sect:1}}

The aim of this section is to prove the following 

\begin{proposition}\label{prop:lebappr}
For every $\eps>0$ there exists a right polygon $\Omega_\eps\subset\subset\Omega$ of side length $r$ such that the $(\Omega_\eps,r)$-interpolation $v:\Omega_\eps\longrightarrow v(\Omega_\eps)\subseteq\R^2$ is $L+\eps$ bi-Lipschitz and satisfies 
\begin{align}
&\Delta_\eps:=v(\Omega_\eps)\subset\subset\Delta\,,\label{luino}\\
&\|v-u\|_{L^{\infty}(\Omega_\eps)}+\|v^{-1}-u^{-1}\|_{L^{\infty}(\Delta_\eps)}+\|Du-Dv\|_{L^p(\Omega_\eps)}+\|Du^{-1}-Dv^{-1}\|_{L^p(\Delta_\eps)}\leq \eps \,, \label{eq:lebappr3}\\
&\LL(\Omega\setminus\Omega_\eps)\leq\eps\,, \quad \LL(\Delta\setminus\Delta_\eps)\leq\eps\,, \quad d(\Omega_\eps,\R^2\setminus\Omega)\geq2r\,,\label{luiue}\\
&\|v-u\|_{L^{\infty}(\Omega_\eps)}\leq \frac{\sqrt{2} r}{6L^3}\,.\label{luiro}
\end{align}
\end{proposition}

The reason why the piecewise affine interpolation of $u$ will be injective on $\Omega_\eps$ is that, for each square $\D_\alpha$ of the $r$-tiling of $\Omega_\eps$, the function $u$ will be uniformly close to an affine $L$ bi-Lipschitz function on the nine squares around $\D_\alpha$. The linear part of each of these affine functions will be the differential of $u$ at some Lebesgue point for $Du$ inside $\D_\alpha$. For this reason, the squares of such $r$ -tiling will be called ``Lebesgue squares''.

The plan of this section is the following. Section~\ref{subsect:11} contains Lemma~\ref{lem:linfty1}, which is the main ingredient in the proof of Proposition~\ref{prop:lebappr}. Indeed, Lemma~\ref{lem:linfty1} says that, when on a square $Du$ is close in average to an $L$ bi-Lipschitz matrix $M$, then $u$ is close in $L^{\infty}$ to an affine function $u_M$ with $Du_M=M$. 
Then, in Section~\ref{subsect:12}, we will determine $\Omega_\eps$ as a suitable union of squares of an $r$-tiling on which Lemma~\ref{lem:linfty1} holds and provides a sufficiently strong $L^{\infty}$ estimate. Finally, in Section~\ref{subs:apprbuoni} we show that the $(\Omega_\eps,r)$-interpolation of $u$ satisfies the required properties. 

\subsection{{An ${L^{\infty}}$ Lemma}}\label{subsect:11}

We are now ready to begin the proof of Proposition~\ref{prop:lebappr}, starting from the following fundamental lemma. Here and in the following, by $\MM(2\times 2;L)$ we denote the set of the two by two invertible matrices such that the affine map $z \mapsto M z$ is $L$ bi-Lipschitz. Moreover, $\Omega$ and $u$ will always be a set and a function as in the assumptions of Theorem~\ref{main}.

\begin{lemma}\label{lem:linfty1}
For any $\eta>0$ there exists $\delta=\delta(\eta)>0$ such that, if $\bz\in\Omega$, $M\in \MM(2\times 2; L)$ and $\rho>0$ are so that $\D(\bz,\rho)\subset\subset\Omega$ and
\begin{equation}\label{eq:dz4r}
\intmed_{\D(\bz,\rho)}|Du(z)-M|\,dz\leq\delta,
\end{equation}
then there exists an affine function $u_M:\R^2\longrightarrow\R^2$ with $Du_M=M$ and such that
\begin{align}\label{eq:linfty1}
|u(z)-u_M(z)|\leq\eta \rho && \forall\,z\in\D(\bz,\rho)\,.
\end{align}
\end{lemma}
\begin{proof}
Up to a translation, we are allowed to assume for simplicity that $\bar z=u(\bar z)=(0,0)\in\R^2$. Let us then call, for a big constant $R$ to be specified later,
\begin{align*}
B^1&:=\bigg\{x\in\big[-\rho/2,\rho/2\big]:\,\int_{-\rho/2}^{\rho/2}|Du(x,t)-M|\,dt\leq \rho R \delta\bigg\}\,,\\
B^2&:=\bigg\{y\in\big[-\rho/2,\rho/2\big]:\,\int_{-\rho/2}^{\rho/2}|Du(t,y)-M|\,dt\leq \rho R \delta\bigg\}\,.
\end{align*}
Notice that, since $u$ is bi-Lipschitz on $\Omega$, then so are its restrictions to the horizontal and vertical segments of the square $\D(0,\rho)$. Hence, the above integrals make sense for every $x$ and $y$. By~\eqref{eq:dz4r} and Fubini--Tonelli Theorem, we readily obtain
\begin{align}\label{piccolo}
\H^1\Big( \big[ -\rho/2, \rho/2\big]\setminus B^1\Big)\leq \frac \rho R\,, && \H^1\Big( \big[ -\rho/2, \rho/2\big]\setminus B^2\Big)\leq \frac \rho R\,.
\end{align}
Define now $u_M(z) = M z$, and $\varphi(z)= u(z) - u_M(z)$. For any $x_1,\, x_2\in B^1$ and $y_1,\, y_2\in B^2$ we immediately get
\begin{equation}\label{firste}\begin{split}
\big| \varphi(x_1,y_1) - \varphi(x_2,y_2)\big|
&\leq \big| \varphi(x_1,y_1) - \varphi(x_2,y_1) \big|+ \big| \varphi(x_2,y_1) - \varphi(x_2,y_2) \big|\\
&\leq\int_{x_1}^{x_2} \big|Du(t,y_1)-M\big|\,dt+\int_{y_1}^{y_2} \big|Du(x_2,t)-M\big|\,dt \leq 2\rho R \delta\,.
\end{split}\end{equation}
Let now $(x,y)\in\D(\bar z,\rho)$ be a generic point. By~(\ref{piccolo}), there exist $x_1\in B^1$ and $y_1\in B^2$ so that
\begin{align*}
\big| x - x_1\big| \leq \frac{\rho}{R}\,, && \big| y - y_1\big| \leq \frac{\rho}{R}\,,
\end{align*}
and since $u$ and $u_M$ are $L$ bi-Lipschitz, thus $\varphi$ is $2L$-Lipschitz, we get
\begin{equation}\label{seconde}
\big|\varphi(x,y) - \varphi(x_1,y_1) \big|\leq \frac{2 \sqrt{2} \rho L}{R}\,.
\end{equation}
Let finally $(x,y)$ and $(\tilde x,\tilde y)$ be two generic points in $\D(\bz,r)$. Putting together~(\ref{firste}) and~(\ref{seconde}) we immediately get
\[
\big| \varphi(x,y) - \varphi(\tilde x,\tilde y)\big| \leq \frac{4 \sqrt{2} \rho L}{R} + 2 \rho R \delta \leq \eta \rho\,,
\]
where the last inequality is true up to take $R$ big enough and then $\delta$ small enough. Since $\varphi(0,0)=0$, this concludes the proof.
\end{proof}

\subsection{{A large right polygon made of Lebesgue squares}}\label{subsect:12}

In this section we show that, for any $\eta>0$, it is possible to construct a right polygon $\Omega_\eta\subset\subset\Omega$ of side length $r_\eta$ such that $\LL(\Omega\setminus\Omega_\eta)\leq\eta$ and such that, for any square $\D(z,r_\eta)$ of the $r_\eta$-tiling of $\Omega_\eta$, the assumption~(\ref{eq:dz4r}) of Lemma~\ref{lem:linfty1} holds on the bigger square $\D(z,3 r_\eta)$. As we will show in Section~\ref{subs:apprbuoni}, if we choose $\eta$ and then $r_\eta$ small enough, the corresponding $(\Omega_\eta, r_\eta)$-interpolation of $u$ satisfies the requirements of Proposition~\ref{prop:lebappr}. Then, $\Omega_\eta$ will turn out to be the right polygon of Lebesgue squares we are looking for. The goal of this section is to show the following estimate.

\begin{lemma}\label{lem:quadbuoni}
For every $\eta>0$ there exists a constant $r=r(\eta)>0$ and an $r$-polygon $\Omega_\eta\subset\subset\Omega$ such that $\LL\big(\Omega\setminus\Omega_\eta\big)\leq\eta$ and each square of the $r$-tiling $\{\D(z_{\alpha},r)\}_{\alpha\in\II(r)}$ satisfies the following properties,
\begin{align}
&\D(z_\alpha,3r)\subset\subset\Omega\quad\forall\,\alpha\in\II(r)\,,\label{1buoni}\\
&\intmed_{\D(z_\alpha,3r)}|Du(z)-M|\,dz\leq\delta(\eta)\text{ for some $M=M(\alpha)\in\MM(2\times2;L)$}\,.\label{1buonibis}
\end{align}
\end{lemma}
\begin{proof}
We start selecting some $r_0=r_0(\eta)>0$ and an $r_0$-polygon $\Omega_0\subset\subset \Omega$ such that $\LL\big(\Omega\setminus\Omega_0\big)\leq \eta/2$ and each square of the $r_0$-tiling of $\Omega_0$ satisfies~(\ref{1buoni}). Then, for every $r$ such that $r_0 \in r \N$, we can regard $\Omega_0$ also as an $r$-polygon, and consequently call $\{\D(z_\alpha,r)\}_{\alpha\in\II_0(r)}$ its $r$-tiling. We define the set
\[
\II(r):=\Big\{\alpha\in\II_0(r):\,\intmed_{\D(z_\alpha,3r)} |Du-M|\leq\delta\text{ for some $M=M(\alpha)\in\MM(2\times2;L)$}\Big\},
\]
where $\delta=\delta(\eta)$ is given by Lemma~\ref{lem:linfty1}, and we let
\[
\Omega_\eta:=\bigcup_{\alpha\in\II(r)} \D(z_\alpha,r)\,.
\]
Since property~(\ref{1buonibis}) is true by construction, to conclude the proof it is enough to select a suitable $r=r(\eta)$ in such a way that $\LL\big(\Omega_0\setminus\Omega_\eta\big)\leq\eta/2$.\par
To do so, we apply the Lebesgue Differentiation Theorem to the map $Du$ finding that, for $\LL$-a.e. $z\in\Omega_0$, there exists $r(z)>0$ such that $\D(z,4r(z))\subseteq \Omega_0$ and
\begin{align*}
\intmed_{\D(z,\rho)}\big|Du(w)-Du(z)\big|\,dw\leq\frac{\delta}{2} && \forall\,0<\rho\leq 4r(z)\,.
\end{align*}
We can then choose $r=r(\eta)$ so small that the set $A(r):=\big\{ z\in \Omega_0 :\, r(z) \leq r\big\}$ satisfies
\begin{equation}\label{Arsmall}
\LL\big(A(r)\big)\leq \eta/2\,.
\end{equation}
We now claim that, for each $\alpha\in \II_0(r)$,
\begin{equation}\label{imply}
\D(z_\alpha,{r})\not\subseteq  A(r) \qquad\Longrightarrow\qquad\alpha\in \II(r).
\end{equation}
Indeed, letting $M=Du(z)$ for some $z\in\D(z_\alpha,r)\setminus A(r)$, by definition of $A(r)$ and $r(z)$ we get
\[\begin{split}
\intmed_{\D(z_\alpha,3r)}|Du-M|& = \frac{1}{9r^2} \int_{\D(z_\alpha,3r)}|Du-M|
\leq \frac{1}{9r^2} \int_{\D(z,4r)}|Du-M|
= \frac{16}{9} \intmed_{\D(z,4r)}|Du-M|\\
&\leq \frac{8}{9} \,\delta\,,
\end{split}\]
thus~(\ref{imply}) is obtained. As a consequence, by~(\ref{Arsmall}) we have that
\[
\LL\Big(\Omega_0\setminus\Omega_\eta\Big) = 
\LL\bigg(\bigcup\nolimits_{\alpha\in\II_0(r)\setminus\II(r)} \D(z_\alpha,r)\bigg)
\leq \LL\big( A(r)\big)\leq \frac \eta 2
\]
and, as we noticed above, this concludes the proof.
\end{proof}

\subsection{Affine approximation of ${u}$ on Lebesgue squares}\label{subs:apprbuoni}
In this section we complete the proof of Proposition~\ref{prop:lebappr}. At this point the proof reduces to show that, provided we choose $\eta$ small enough, the $(\Omega_\eta, r)$-interpolation of $u$ on the right polygon $\Omega_\eta$ as in Lemma~\ref{lem:quadbuoni} satisfies the properties of Proposition~\ref{prop:lebappr}.

\begin{proof}[Proof of Proposition~\ref{prop:lebappr}: ]
Let $\eps>0$ be a given constant. Then, let $\eta=\eta(\eps)$ be a sufficiently small constant, whose value will be precised later. Define now $\delta=\delta\big(\eta(\eps)\big)$ as in Lemma~\ref{lem:linfty1}, and define also $r=r\big(\eta(\eps)\big)$ and $\Omega_\eps=\Omega_{\eta(\eps)}$ according to Lemma~\ref{lem:quadbuoni}. We will show that the right polygon $\Omega_\eps$ fulfills all the requirements of the proposition as soon as $\eta(\eps)$ is small enough. To this aim we call, as in the statement, $v:\Omega_\eps\longrightarrow \Delta_\eps$ the $(\Omega_\eps, r)$-interpolation of $u$ (see Definition~\ref{def:uinterp}) on the right polygon $\Omega_\eps$.\par

Let us briefly fix some notation which will be used through the proof. For any $\alpha\in \II(r)$, we define $M_\alpha\in\MM(2\times 2;L)$ so that~(\ref{1buonibis}) holds. Applying Lemma~\ref{lem:linfty1} with $\rho=3r$, we get an affine function $u_\alpha:\R^2\longrightarrow \R^2$ such that $Du_\alpha=M_\alpha$ and
\begin{equation}\label{frominf}
\big|u-u_\alpha\big|\leq 3 \eta r \qquad \hbox{on } \D(z_\alpha,3r)\,.
\end{equation}
Figure~\ref{Fig:sara} depicts the functions $u$, $v$ and $u_\alpha$.

\begin{figure}[htbp]
\begin{center}
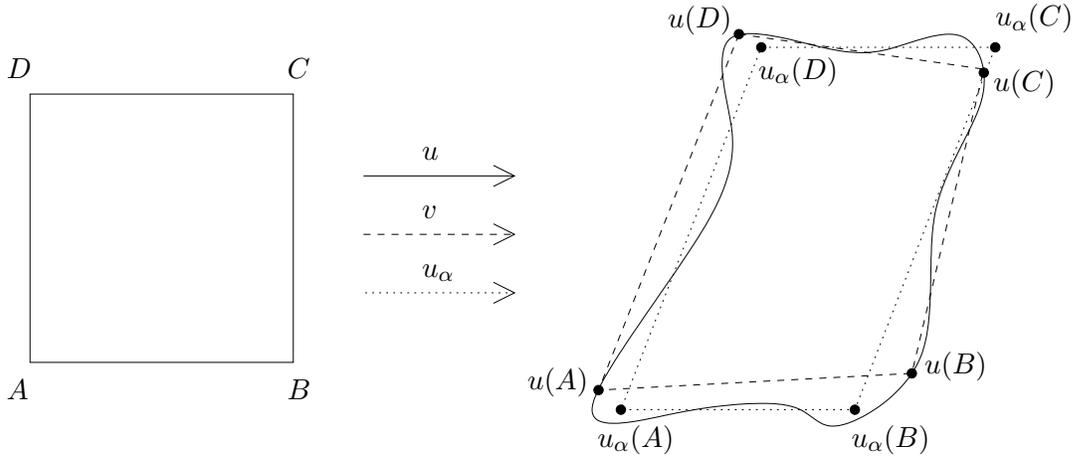
\caption{The functions $u$, $v$ and $u_\alpha$ on a square.}\label{Fig:sara}
\end{center}
\end{figure}

We can then start the proof, which will be divided in some steps for clarity.

\step{I}{For any $\alpha\in\II(r)$, $v\big(\D(z_\alpha,r)\big)\subseteq u\big(\D(z_\alpha, 3 r)\big)$.}

Take $\alpha\in\II(r)$. Keeping in mind~(\ref{frominf}) and recalling the definition of $v$, we get that
\begin{equation}\label{tra1}
v\big( \D(z_\alpha,r)\big) \subseteq \B\Big(u_\alpha\big(\D(z_\alpha,r)\big),3\eta r\Big)\,,
\end{equation}
where for any set $X\subseteq \R^2$ we denote by $\B(X,r)$ the $r$-neighborhood of $X$. Similarly, we get that
\[
u\big(\D(z_\alpha, 3r )\big) \supseteq \Big\{ x:\, \B(x,3\eta r)\subseteq u_\alpha\big(\D(z_\alpha, 3r)\big)\Big\}\,.
\]
Hence, the step is concluded if
\[
\B\Big(u_\alpha\big(\D(z_\alpha,r)\big),6\eta r\Big)\subseteq u_\alpha\big(\D(z_\alpha, 3r)\big)\,,
\]
which in turn, recalling that $Du_\alpha \equiv M_\alpha \in \MM(2\times 2;L)$, is true as soon as $\eta< (6L)^{-1}$.\par

Observe that, as an immediate consequence of this step and~(\ref{1buoni}), we have $\Delta_\eps\subset \subset\Delta$, that is, (\ref{luino}) holds.

\step{II}{Injectivity of $v$.}

Take $\alpha\in\II(r)$. Applying again~(\ref{frominf}) as in Step~I, we deduce that $v$ is injective on $\D(z_\alpha,3r)\cap \Omega_\eps$ as soon as $\eta< (6L)^{-1}$. To conclude that $v$ is injective, then, we have to show that $v\big(\D(z_\alpha,r)\big)\cap v\big(\D(z_\beta,r)\big)=\emptyset$ if $\D(z_\alpha,r)$ and $\D(z_\beta,r)$ are two non-adjacent squares of the tiling of $\Omega_\eps$. And in fact, if $\D_\alpha$ and $\D_\beta$ are non-adjacent, then the fact that $v\big(\D(z_\alpha,r)\big)\cap v\big(\D(z_\beta,r)\big)=\emptyset$ for a small $\eta$ follows as an immediate consequence of~(\ref{frominf}) and~(\ref{tra1}) arguing as in Step~I.

\step{III}{Estimate on $\|v-u\|_{L^{\infty}(\Omega_\eps)}$ and on $\|v^{-1}-u^{-1}\|_{L^\infty(\Delta_\eps)}$.}

Fix a generic square $\D_\alpha$ of the $r$-tiling of $\Omega_\eps$, and observe that $\|u_\alpha-u\|_{L^{\infty}(\D_\alpha)}\leq 3\eta r$ by~(\ref{frominf}). Moreover, $v$ and $u_\alpha$ are both affine on each of the two right triangles on which $\D_\alpha$ is divided, and since on the vertices of these triangles $v$ equals $u$, again by~(\ref{frominf}) we deduce also $\|v-u_\alpha\|_{L^{\infty}(\D_\alpha)}\leq 3\eta r$. Thanks to these two estimates, we deduce
\begin{equation}\label{est1}\begin{split}
\|v-u\|_{L^\infty(\Omega_\eps)} &= \sup_{\alpha\in\II(r)} \|v-u\|_{L^{\infty}(\D_\alpha)}
\leq \sup_{\alpha\in\II(r)}\|v-u_\alpha\|_{L^{\infty}(\D_\alpha)} + \|u_\alpha-u\|_{L^{\infty}(\D_\alpha)}\\
&\leq 6 \eta r 
\leq \frac{\eps}{4L}\,,
\end{split}\end{equation}
where the last inequality is true as soon as $\eta$, hence also $r$, is small enough.\par
Since we have already proven that $v$ is injective, the $L^{\infty}$ estimate for the inverse maps is now a simple consequence. Indeed, taking a generic point $w=v(z)\in\Delta_\eps$, with $z\in\Omega_\eps$, by~(\ref{est1}) we have
\[
\big|u^{-1}(w)-v^{-1}(w)\big|=\big|u^{-1}(v(z))-u^{-1}(u(z))\big| \leq L \big| v(z)-u(z)\big|\leq \frac \eps 4\,,
\]
so that
\begin{equation}\label{est2}
\|u^{-1}- v^{-1} \|_{L^\infty(\Delta_\eps)} \leq \frac \eps 4\,.
\end{equation}

\step{IV}{Estimate on $\|Dv-Du\|_{L^p(\Omega_\eps)}$.}

Let us start observing that, since by construction $|Du|\leq L$ and $|Dv|\leq \sqrt 2 L$, one has
\begin{equation}\label{eq:2}
\begin{split}
\|Dv-Du\|^p_{L^p(\Omega_\eps)}&=\sum_{\alpha\in\II(r)} \|Dv-Du\|^p_{L^p(\D_\alpha)}\\
& \leq \big( 3L\big)^{p-1} \sum_{\alpha\in\II(r)} \|Dv-Du\|_{L^1(\D_\alpha)}\\
&\leq \big( 3L\big)^{p-1} \sum_{\alpha\in\II(r)} \Big(\|Dv-Du_\alpha\|_{L^1(\D_\alpha)}+\|Du_\alpha-Du\|_{L^1(\D_\alpha)}\Big)\,.
\end{split}\end{equation}
By~(\ref{1buonibis}), we already know that for each $\alpha\in\II(r)$
\begin{equation}\label{ins1}
\|Du-Du_\alpha\|_{L^1(\D_\alpha)} = \int_{\D(z_\alpha, r)} \big|Du - Du_\alpha \big| \leq 9r^2 \intmed_{\D(z_\alpha, 3r)} \big|Du - M_\alpha \big| \leq 9 \delta r^2 = 9\delta \big|\D_\alpha\big|\,.
\end{equation}
Let us then concentrate on $\|Dv -Du_\alpha \|_{L^1(\D_\alpha)}$. Consider the triangle $T=z_1z_2z_3$, being
\begin{align*}
z_1\equiv z_\alpha + \big(-r/2, -r/2 \big)\,, &&z_2\equiv z_\alpha + \big(r/2, -r/2 \big)\,, &&z_3\equiv z_\alpha + \big(r/2, r/2 \big)\,.
\end{align*}
Since both $v$ and $u_\alpha$ are affine on $T$, then in particular $Dv - Du_\alpha$ is a constant linear function on $T$. Recalling again~(\ref{frominf}), let us then calculate
\[\begin{split}
\Big|\big(Dv_{|T} - Du_\alpha\big) (r \ee_1)\Big| &=  \Big| \big(v(z_2) - v(z_1)\big) -\big(u_\alpha(z_2)-u_\alpha(z_1) \big)\Big| \\
&=  \Big| \big(u(z_2) - u(z_1)\big) -\big(u_\alpha(z_2)-u_\alpha(z_1) \big)\Big| \leq 6 \eta r\,,
\end{split}\]
and similarly
\[
\Big|\big(Dv_{|T} - Du_\alpha\big) (r \ee_2)\Big|
= \Big| \big(v(z_3) - v(z_2)\big) -\big(u_\alpha(z_3)-u_\alpha(z_2) \Big| \leq 6\eta r\,.
\]
We deduce that $\|Dv - Du_\alpha\|_{L^\infty(T)}\leq 6\sqrt 2 \eta$. We can argue in the same way for all the different triangles in which $\D(z_\alpha,3r)\cap\Omega_\eps$ is divided, thus getting
\begin{equation}\label{ins2}
\big\|Dv - Du_\alpha\big\|_{L^\infty(\D(z_\alpha,3r)\cap \Omega_\eps)}\leq 6\sqrt 2 \eta\leq 9\eta\,.
\end{equation}
Inserting this estimate and~(\ref{ins1}) into~(\ref{eq:2}), we get
\begin{equation}\label{est3}
\|Dv-Du\|^p_{L^p(\Omega_\eps)}\leq \big( 3L\big)^{p-1} 9\big(\delta+\eta) \sum_{\alpha\in\II(r)} \big|\D_\alpha\big|
=\big( 3L\big)^{p-1} 9 \big(\eta+\delta\big) \big|\Omega_\eps\big| \leq \bigg(\frac \eps 4\bigg)^p\,,
\end{equation}
where again the last inequality holds true as soon as $\eta$, hence also $\delta$, is small enough.

\step{V}{Bi-Lipschitz estimate for $v$.}

Take a point $z\in \D(z_\alpha,3r)\cap\Omega_\eps$. Recalling that $u_\alpha$ is $L$ bi-Lipschitz and~(\ref{ins2}), we get
\begin{equation}\label{princ}
\frac 1L - 9\eta \leq \big|Dv(z)\big| \leq L + 9\eta\,.
\end{equation}
Let then $z,\,z'\in\Omega_\eps$ be two generic points, and assume that $z\in\D(z_\alpha,r)$. If one has $z'\in\D(z_\alpha,3r)$, then an immediate geometric argument using the definition of $v$ and~(\ref{princ}) yields
\begin{equation}\label{bil1}
\bigg(\frac 1L - 9\eta\bigg) \big|z-z'\big| \leq \big|v(z) - v(z')\big| \leq \big(L+9\eta\big) \big|z-z'\big|\,.
\end{equation}
On the other hand, assume that $z'\notin \D(z_\alpha,3r)$, so that $|z-z'|\geq r$. In this case, the $L^\infty$ estimate~(\ref{est1}) gives
\begin{equation}\label{bil2}
\big|v(z) - v(z')\big| \leq \big|u(z) - u(z')\big|+\big|v(z) - u(z)\big|+\big|v(z') - u(z')\big| \leq \big( L+12 \eta\big) \big| z - z'\big|\,,
\end{equation}
and similarly
\begin{equation}\label{bil3}
\big|v(z) - v(z')\big|\geq\big|u(z)-u(z')\big|-\big|v(z) - u(z)\big|-\big|v(z') - u(z')\big| \geq \bigg(\frac 1L- 12 \eta  \bigg) \big| z - z'\big|\,.
\end{equation}
Putting together~(\ref{bil1}), (\ref{bil2}) and~(\ref{bil3}), provided that $\eta$ is small enough we conclude that $v$ is $L+\eps$ bi-Lipschitz.

\step{VI}{Estimate on $\|Dv^{-1}-Du^{-1}\|_{L^{p}(\Delta_\eps)}$.}

Fix a generic $\alpha\in\II(r)$, and recall the elementary fact that, given two invertible matrices $A$ and $B$, one always has $\big|B^{-1} - A^{-1}\big| \leq \big| A^{-1}\big| \big| B^{-1}\big| \big| B - A\big|$. Since by construction $u$ and $u_\alpha$ are $L$ bi-Lipschitz, and $Du_\alpha$ is constant on $\D_\alpha$, then Step~I, (\ref{1buonibis}) and~(\ref{ins1}) ensure that
\[\begin{split}
\|Du^{-1}&-Du^{-1}_\alpha\|_{L^1(v(\D_\alpha))}=
\int_{v(\D(z_\alpha,r))} \big| Du^{-1}(z) - Du^{-1}_\alpha(z)\big|\,dz\\
&\leq L^2 \int_{u(\D(z_\alpha,3r))} \big| Du\big(u^{-1}(z)\big) - M_\alpha\big|\,dz
\leq L^4 \int_{\D(z_\alpha,3r)} \big| Du(w) - M_\alpha\big|\,dw\\
&= 9r^2 L^4\intmed_{\D(z_\alpha, 3r)} \big| Du - M_\alpha\big|
\leq 9r^2 L^4 \delta = 9 L^4 \delta \, \big|\D_\alpha\big|\,.
\end{split}\]
On the other hand, again using $\big|B^{-1} - A^{-1}\big| \leq \big| A^{-1}\big| \big| B^{-1}\big| \big| B - A\big|$, the fact that $Du_\alpha$ is constant on $\D_\alpha$, the fact that $u_\alpha$ is $L$ bi-Lipschitz by definition while $v$ is $(L+\eps)$ bi-Lipschitz by Step~V, and~(\ref{ins2}), we readily obtain
\[
\big\|Dv^{-1} - Du_\alpha^{-1}\big\|_{L^\infty(v(\D_\alpha))}\leq  L (L+\eps) 9\eta\leq 18 L^2 \eta\,.
\]
We can then repeat the same argument as in~(\ref{eq:2}) to get
\begin{equation}\label{est4}\begin{split}
\hspace{-80pt}
\|Dv^{-1}-Du^{-1}\|^p_{L^p(\Delta_\eps)} 
&\leq \big( 3L\big)^{p-1} \sum_{\alpha\in\II(r)} \|Dv^{-1}-Du^{-1}_\alpha\|_{L^1(v(\D_\alpha))}+\|Du^{-1}_\alpha-Du^{-1}\|_{L^1(v(\D_\alpha))} 
\hspace{-80pt}\\
&\leq  \big(3L\big)^{p-1} \Big( 18 L^2 \eta \big| \Delta_\eps\big|  + 9L^4 \delta \big| \Omega_\eps\big|\Big)
\leq \bigg(\frac \eps 4\bigg)^p\,,
\end{split}\end{equation}
where as usual the last estimate holds possibly decreasing $\eta$ and then also $\delta$.

\step{VII}{Conclusion.}

Let us now conclude the proof of Proposition~\ref{prop:lebappr} by checking that $\Omega_\eps$ fulfills all the requirements of the statement. The fact that $v$ is $L+\eps$ bi-Lipschitz is given by Step~V. The validity of~(\ref{luino}) has been observed in Step~I. The estimate~(\ref{eq:lebappr3}) just follows adding~(\ref{est1}), (\ref{est2}), (\ref{est3}) and~(\ref{est4}). Concerning~(\ref{luiue}), the facts that $\Omega\setminus\Omega_\eps$ is small and that $d(\Omega_\eps,\R^2\setminus\Omega)\geq 2r$ are given by Lemma~\ref{lem:quadbuoni}, while the fact that also $\Delta\setminus\Delta_\eps$ is small is immediate by the bi-Lipschitz property of $u$ and the $L^\infty$ estimate~(\ref{eq:linfty1}) of Lemma~\ref{lem:linfty1}. Finally, (\ref{luiro}) is immediate provided that we choose $\eta\leq \sqrt 2/(36 L^3)$, since~(\ref{est1}) ensures that $\|v-u\|_{L^\infty(\Omega_\eps)}\leq 6\eta r$.
\end{proof}

\section{Approximation out of ``Lebesgue squares''\label{sect:2}}

In this section we complete the proof of Theorem~\ref{main}, defining the countably piecewise affine approximation of $u$ out of the large right polygon $\Omega_\eps$ of ``Lebesgue squares'' constructed in Proposition~\ref{prop:lebappr}.
Following the scheme outlined in Section~\ref{sect:scheme}, the construction is carried out in three steps: the covering of $\Omega\setminus\Omega_\eps$ with a suitable (locally finite) tiling, the definition of a bi-Lipschitz piecewise affine approximation of $u$ on the grid of the tiling and, finally, the extension of the approximating function to the interior of the grid by means of Theorem~\ref{teo:bilsquare}. The main result of this section is the following.
\begin{proposition}\label{prop:lipext}
Let $v_\eps:\Omega_\eps\longrightarrow\Delta_\eps$ be a piecewise affine bi-Lipschitz function as in Proposition~\ref{prop:lebappr}. Then, there exists a $C_1 L^4$ bi-Lipschitz countably piecewise affine function $\tilde v_\eps:\Omega\setminus\Omega_\eps\longrightarrow\Delta\setminus\Delta_\eps$, where $C_1$ is a geometric constant, such that $\tilde v_\eps=u$ on $\partial \Omega$ and $\tilde v_\eps=v_\eps$ on $\partial\Omega_\eps$.
\end{proposition}
We can immediately notice that Theorem~\ref{main} follows as an easy consequence of Propositions~\ref{prop:lebappr} and~\ref{prop:lipext}.

\begin{proof}[Proof of Theorem~\ref{main}: ]
Take $\eps>0$, and apply Proposition~\ref{prop:lebappr} to get an $r$-polygon $\Omega_\eps\subset\subset \Omega$ and a piecewise affine bi-Lipschitz function $v_\eps:\Omega_\eps\longrightarrow\Delta_\eps$. By Proposition~\ref{prop:lipext}, we have a $C_1L^4$ bi-Lipschitz function $\tilde v_\eps: \Omega\setminus\Omega_\eps\longrightarrow\Delta\setminus\Delta_\eps$, so we can define the function $v:\Omega\longrightarrow\Delta$ as $v\equiv v_\eps$ on $\Omega_\eps$ and $v\equiv \tilde v_\eps$ on $\Omega\setminus\Omega_\eps$. Since $v_\eps$ (resp., $\tilde v_\eps$) is bi-Lipschitz with constant $L+\eps$ (resp., $C_1 L^4$), and $\tilde v_\eps = v_\eps$ on $\partial \Omega_\eps$, we have that $v$ is a bi-Lipschitz homeomorphism with constant $C_1 L^4$. Moreover, by construction $v$ is countably piecewise affine and it is orientation-preserving, since so is $u$ and $v\equiv u$ on $\partial \Omega$. We are then left with showing that $v$ satisfies~(\ref{approximation}), and by~(\ref{eq:lebappr3}) this basically reduces to consider what happens in $\Omega\setminus\Omega_\eps$. Since $\tilde v_\eps$ is bi-Lipschitz with constant $C_1 L^4$, by~(\ref{luiue}) we clearly have
\begin{equation}\label{main1}
\|Dv-Du\|_{L^p(\Omega\setminus\Omega_\eps)}\leq \|Dv-Du\|_{L^\infty(\Omega\setminus\Omega_\eps)} \big| \Omega\setminus\Omega_\eps\big|^{1/p}
\leq \big(L+ C_1 L^4\big) \eps^{1/p}\,,
\end{equation}
and similarly
\begin{equation}\label{main2}
\|Dv^{-1}-Du^{-1}\|_{L^p(\Delta\setminus\Delta_\eps)} \leq \big(L+ C_1 L^4\big) \eps^{1/p}\,.
\end{equation}
Concerning the $L^\infty$ estimates, since $\big|\Omega\setminus\Omega_\eps\big| \leq \eps$ then for every $z\in \Omega\setminus\Omega_\eps$ there exist $z'\in \Omega_\eps$ such that $|z-z'|\leq \sqrt{\eps/\pi}$, thus by~(\ref{eq:lebappr3}) we find
\[\begin{split}
|v(z)-u(z)| &\leq |v(z)-v(z')| + |v(z')-u(z')| + |u(z')-u(z)| \\
&\leq \big( L + C_1 L^4\big) \sqrt{\frac \eps\pi} + \|v_\eps - u\|_{L^\infty(\Omega_\eps)}
\leq \big( L + C_1 L^4\big) \sqrt{\frac \eps\pi} + \eps\,.
\end{split}\]
Arguing in the same way to bound $|v^{-1}(w)-u^{-1}(w)|$ for a generic $w\in \Delta\setminus\Delta_\eps$ yields
\begin{align}\label{main3}
\|v-u\|_{L^{\infty}(\Omega\setminus\Omega_\eps)} 
\leq \big( L + C_1 L^4\big) \sqrt{\frac \eps\pi} + \eps\,, &&
 \|v^{-1}-u^{-1}\|_{L^{\infty}(\Delta\setminus\Delta_\eps)} 
\leq \big( L + C_1 L^4\big) \sqrt{\frac \eps\pi} + \eps\,.
\end{align}
Putting together~(\ref{main1}), (\ref{main2}) and~(\ref{main3}), we find that $v$ satisfies~(\ref{approximation}) as soon as $\eps$ is chosen small enough, depending on $\bar\eps$. Hence, we have found the countably piecewise affine approximation as required. Concerning the smooth approximation, its existence directly follows applying Theorem~\ref{MP}, thus we have in particular $C_2=70 C_1^{7/3}$.
\end{proof}

We have now to prove Proposition~\ref{prop:lipext}. To do so, let us fix some notation. Recall that $\Omega_\eps$ is an $r$-polygon for some $r=r(\eps)$. We will then start by selecting a suitable tiling $\{\D_j=\D(z_j,r_j)\}_{j\in\N}$ of $(\Omega,\Omega_\eps)$, according with Definition~\ref{def:tiling}. This means that $\{\D_j\}$ is a tiling of $\Omega$ whose restriction to $\Omega_\eps$ coincides with the $r$-tiling of $\Omega_\eps$. The only requirements that we ask to $\{\D_j\}$ are the following,
\begin{flalign}
\quad &r_j = r \quad \forall\,j:\, \clos\D_j \cap \partial \Omega_\eps \neq \emptyset\,, & \label{first}\\
& \D_j \subset\subset \Omega\quad \forall\, j\in\N\,.\label{second}
\end{flalign}
Notice that~(\ref{first}) is possible thanks to~(\ref{luiue}), while~(\ref{second}) basically means that the tiling has to be countable instead of finite, and the squares have to become smaller and smaller when approaching the boundary of $\Omega$. Of course, in the particular case when $\Omega$ itself is a right polygon, instead of~(\ref{second}) one could have asked the tiling to be finite (we will discuss this possibility more in detail in Remark~\ref{remfin}).\par
Since it is of course possible to find a tiling of $(\Omega,\Omega_\eps)$ which satisfies~(\ref{first}) and~(\ref{second}), from now on we fix such a tiling, and we denote by $\Q$ its associated $1$-dimensional grid according to Definition~\ref{def:grid}. Moreover, we set $\Q'=\Q\cap (\Omega\setminus \clos{\Omega_\eps})$, which is the part of the grid on which we really need to work. Notice that $\Q'$ is a $1$-dimensional set, made by all the sides of the grid $\Q$ which lie in $\Omega\setminus\clos\Omega_\eps$.\par
Let us call now $w_\alpha$ the generic vertex of $\Q'$, hence, the generic vertex of the grid $\Q$ which does not belong to $\Omega_\eps$ (however, notice that $w_\alpha$ may belong to $\partial\Omega_\eps$, since by definition $\Q'$ is relatively open around $\partial\Omega_\eps$). Each vertex $w_\alpha$ is of the form $w_\alpha = z_j + (\pm r_j/2, \pm r_j/2)$ for some $j$, and it is one extreme of either three, or four sides of $\Q$. To shorten the notation, we will denote the other extremes of these sides by $w_\alpha^i$ with $1\leq i \leq \bar i(\alpha)$, being then $\bar i(\alpha)\in \{3,\,4\}$. Finally, we will denote by $\ell_\alpha$ the minimum of the lengths of the sides $w_\alpha w_\alpha^i$. Observe that if $w_\alpha\notin\partial\Omega_\eps$, then $w_\alpha$ is one extreme of either three or four sides of $\Q'\subseteq \Q$. On the other hand, if $w_\alpha\in\partial\Omega_\eps$, then by~(\ref{first}) it is one extreme of four sides of $\Q$, either one or two of these four sides lies in $\Q'$, and $\ell_\alpha=r$.\par

Thanks to Theorem~\ref{teo:bilsquare}, to obtain the piecewise affine function $\tilde v_\eps$ of Proposition~\ref{prop:lipext} we essntially have to define it, in a suitable way, on the $1$-dimensional grid $\Q'$. To do so, our main ingredients are the following two lemmas. The first one (Lemma~\ref{lem:curves}) states that, on any given segment of $\Omega$, $u$ can be approximated as well as desired in $L^{\infty}$ with suitable piecewise affine $3L$ bi-Lipschitz functions. This is of course of primary importance to define the piecewise affine approximation $\tilde v_\eps$ of $u$ on the sides of the grid $\Q'$, but it is still not enough. In fact, we have to take some additional care to treat the ``crosses'' of $\Q'$ (that is, the regions around the vertices), in order to be sure that our affine $\tilde v_\eps$ on $\Q'$ remains injective. This will be obtained thanks to the second Lemma~\ref{lem:cross}.\par

To state the next two lemmas, it will be useful to introduce some piece of notation.

\begin{definition}[Interpolation of $u$]\label{def:interp}
Given a segment $pq\subset\subset \Omega$, let $\{z_i z_{i+1}\}_{0\leq i < N}$ be $N$ essentially disjoint segments whose union is $pq$, with $z_0=p$ and $z_N=q$. For any such subdivision of the segment, will call \emph{interpolation of $u$} the finitely piecewise affine function $u_{pq}:pq\longrightarrow \R^2$ such that, for any $0\leq i \leq N-1$ and any $0\leq t \leq 1$,
\[
u_{pq}\Big(z_{i}+t\big(z_{i+1}-z_i\big)\Big)=u(z_{i})+t\big(u(z_{i+1})-u(z_i)\big)\,.
\]
\end{definition}

\begin{definition}[Adjusted function and crosses]\label{eq:valpha}
Let $\{\xi_\alpha\}_{\alpha\in\N}$ be a sequence such that for any $\alpha$ one has $3 L \xi_\alpha \leq \ell_\alpha$. For any $\alpha\in\N$ and any $1\leq i\leq \bar i(\alpha)$, we define $\xi_\alpha^i$ as the biggest number such that
\[
\left\{\begin{array}{ll}
\Big|u(w_\alpha)-u\big(w_\alpha + \xi_\alpha^i(w^i_\alpha-w_\alpha)\big)\Big| \leq \xi_\alpha \qquad &\hbox{if $w_\alpha w_\alpha^i\subset\Q'$}\,,\\[5pt]
\Big|u(w_\alpha)-v_\eps\big(w_\alpha + \xi_\alpha^i(w^i_\alpha-w_\alpha)\big)\Big| \leq \xi_\alpha \qquad &\hbox{if $w_\alpha w_\alpha^i\subset \Q\setminus\Q'$}\,.\\
\end{array}\right.
\]
We will call \emph{adjusted function} the function $\uad:\Q\longrightarrow\R^2$ defined as follows. First of all, we set $\uad=v_\eps$ on $\Q\setminus \Q'$. Then, let $w_\alpha w_\beta$ be a side of $\Q'$, thus being $w_\beta=w_\alpha^i$ and $w_\alpha=w_\beta^j$ for two suitable $i$ and $j$. We define
\[
\uad\big(w_\alpha + t(w_\beta-w_\alpha)\big):=\left\{\begin{array}{ll}
u(w_\alpha)+\frac{t}{\xi_\alpha^i}\Big(u\big(w_\alpha+\xi_\alpha^i(w_\beta-w_\alpha)\big)-u(w_\alpha)\Big)&\hbox{in $(0,\xi_\alpha^i)$\,,}\\
u\big(w_\alpha + t(w_\beta-w_\alpha)\big) &\hbox{in $(\xi_\alpha^i,1-\xi_\beta^j)$\,,}\\
u(w_\beta)+\frac {(1-t)}{\xi_\beta^j}\Big(u\big(w_\beta+\xi_\beta^j (w_\alpha-w_\beta)\big)-u(w_\beta)\Big) &\hbox{in $(1-\xi_\beta^j,1)$.}
\end{array}\right.
\]
In words, for any side in $\Q'$, $\uad$ coincides with $u$ in the internal part of the side, while the two parts closest to the vertices $w_\alpha$ and $w_\beta$ are replaced with segments. Moreover, for any vertex $w_\alpha$ of $\Q'$ we will define its associated \emph{cross} as
\[
Z_\alpha = \bigcup_{i=1}^{\bar i(\alpha)} \Big\{ w_\alpha + t (w_\alpha^i - w_\alpha):\, 0\leq t \leq \xi_\alpha^i\Big\}\,.
\]
\end{definition}

\begin{remark}\label{rightafter}
Some remarks are now in order. First of all, since $u$ is $L$ bi-Lipschitz on $\Omega$, and also $v_\eps$ is $L$ bi-Lipschitz on any segment $w_\alpha w_\alpha^i\subseteq \Q\setminus\Q'$, by the choice $3 L \xi_\alpha \leq \ell_\alpha$ we directly deduce that $0<\xi_\alpha^i\leq 1/3$ for any $\alpha$ and any $i$. Thus, two different crosses have always empty intersection. For the same reason, each of the $\bar i(\alpha)$ extremes of the cross $Z_\alpha$ has a distance at least $\xi_\alpha/L$ from $w_\alpha$. Finally, for all different $\alpha$ and $\beta$ one has $\B(u(w_\alpha),\xi_\alpha)\cap \B(u(w_\beta),\xi_\beta)=\emptyset$. Indeed, assuming without loss of generality that $\ell_\alpha \geq \ell_\beta$, we have $\big|u(w_\beta)-u(w_\alpha) \big|\geq \ell_\alpha/L$. And as a consequence, $\xi_\alpha+\xi_\beta\leq \ell_\alpha/(3L) + \ell_\beta/(3L)\leq 2\ell_\alpha/(3L)<\big|u(w_\beta)-u(w_\alpha) \big|$.
\end{remark}

\begin{lemma}\label{lem:curves}
For every segment $pq\subset\subset\Omega$ and every $\delta>0$, there exists a function $\upqd:pq\longrightarrow \Delta$ which is a $4L$ bi-Lipschitz interpolation of $u$ with the property that $\|\upqd-u\|_{L^{\infty}(pq)}\leq\delta$.
\end{lemma}

\begin{lemma}\label{lem:cross}
There exists a sequence $\{\xi_\alpha\}_{\alpha\in\N}$ such that the associated adjusted function $\uad:\Q\longrightarrow \R^2$ is $18L$ bi-Lipschitz and $\uad(\Q)\subseteq \Delta$.
\end{lemma}

Before giving the proof of Lemmas~\ref{lem:curves} and~\ref{lem:cross}, we show how they enter into the proof of Proposition~\ref{prop:lipext}.

\begin{proof}[Proof of Proposition~\ref{prop:lipext}: ]
To define the searched function $\tilde v_\eps: \Omega\setminus\Omega_\eps\longrightarrow \Delta\setminus\Delta_\eps$, let first $\uad:\Q\longrightarrow\R^2$ be an adjusted function according with Lemma~\ref{lem:cross}, corresponding to the sequence $\{\xi_\alpha\}$. Our strategy will be first to define a suitable piecewise affine and injective function $\uad':\Q\longrightarrow\Delta$, coinciding with $\uad$ near the vertices $w_\alpha$, and then to obtain $\tilde v_\eps$ extending $\uad'$ in the interior of each square making use of Theorem~\ref{teo:bilsquare}. We divide the proof in some steps.

\step{I}{Definition of $\uad':\Q\longrightarrow\Delta$.}
First of all, we define $\uad'=\uad=v_\eps$ on $\Q\setminus \Q'$. Then, we consider a generic side $w_\alpha w_\beta\subseteq \Q'$ and define $pq$ the \emph{internal segment} of the side $w_\alpha w_\beta$, that is, $p$ and $q$ are the extremes of the segment $w_\alpha w_\beta\setminus \big( Z_\alpha \cup Z_\beta\big)$. Taking now a small constant $\delta=\delta(\alpha,\,\beta)$, to be fixed later, we set $\uad' = \uad$ on $w_\alpha w_\beta \cap \big(Z_\alpha \cup Z_\beta\big)$, and $\uad'=\upqd$ on $pq$, where $\upqd$ is given by Lemma~\ref{lem:curves}.\par

By definition, it is clear that $\uad'$ is a continuous, countably piecewise affine function on $\Q$. Moreover, since the different constants $\delta(\alpha,\,\beta)$ can be chosen arbitrarily small, each one independently of the others, and any internal segment $pq$ is compactly supported in $\Omega$, one can clearly assume that $\uad'(\Q) \subseteq \Delta$. In addition, since $\uad'$ is obtained gluing the $4L$ bi-Lipschitz functions $\upqd$ and the $18L$ bi-Lipschitz function $\uad$, we clearly have that $\uad'$ is $18\sqrt 2L$-Lipschitz (but, {\it a priori}, not bi-Lipschitz and not even injective!). To conclude the proof, we will then show that in fact $\uad'$ is bi-Lipschitz (hence, in particular, injective), and eventually we will extend $\uad'$ to the interior of the squares of the tiling (hence, to the whole $\Omega\setminus\Omega_\eps$).\par

Let us then fix two points $z,\, z' \in \Q$. In the next Steps~II--IV we will show that, provided that the constants $\delta(\alpha,\,\beta)$ are chosen small enough,
\begin{equation}\label{bilpr}
\big|\uad'(z)-\uad'(z')\big| \geq \frac 1{72L} |z-z'|\,.
\end{equation}
We consider separately the different possible reciprocal positions of $z$ and $z'$.

\step{II}{The case in which $z\in pq\subseteq w_\alpha w_\beta,\, z'\notin w_\alpha w_\beta$\,.}

In this case, as observed in Remark~\ref{rightafter}, we know that $|z-z'|\geq \xi_\alpha/L$. Thus, there are two subcases. If $z'$ does not belong to any internal segment (hence, either $z'$ belongs to some cross, or $z'\in \Omega_\eps$), then $\uad'(z')=\uad(z')$ and then by Lemma~\ref{lem:cross}, provided that we choose
\begin{equation}\label{bounddelta}
\delta(\alpha,\,\beta) \leq \frac{\min\{\xi_\alpha,\,\xi_\beta\}}{36 L^2}\,,
\end{equation}
we have
\[\begin{split}
\big|\uad'(z)-\uad'(z')\big| &= \big|\upqd(z)-\uad(z')\big|
\geq  \big|\uad(z)-\uad(z')\big| - \big|\upqd(z)-\uad(z)\big|\\
&=  \big|\uad(z)-\uad(z')\big| - \big|\upqd(z)-u(z)\big|
\geq  \frac{1}{18L}\,|z-z'\big| - \delta(\alpha,\,\beta)\\
&\geq  \frac{1}{18L}\,|z-z'\big| - \frac{\xi_\alpha}{36 L^2}
\geq  \frac{1}{36L}\,|z-z'\big|\,,
\end{split}\]
so that~(\ref{bilpr}) is proved.\par
Consider now the other subcase, namely, when $z'$ belongs to some other internal segment $p'q'\subseteq w_{\alpha'}w_{\beta'}$. In that case, since by construction and~(\ref{bounddelta}) it is $|z-z'|\geq 36 L\delta(\alpha,\,\beta)$ and $|z-z'|\geq 36 L\delta(\alpha',\,\beta')$, one directly has
\[\begin{split}
\big|\uad'(z)-\uad'(z')\big| &= \big| \upqd(z)  - u^{\delta'}_{p'q'}(z')\big|
\geq \big| u(z) - u(z') \big| - \big| u(z) - \upqd(z) \big| - \big| u(z') - u^{\delta'}_{p'q'}(z') \big| \\
&\geq \frac 1L\, | z -z' \big| - \delta(\alpha,\,\beta) - \delta(\alpha',\,\beta')
\geq \frac {17}{18L} |z-z'|\,,
\end{split}\]
hence again~(\ref{bilpr}) is established.

\step{III}{The case in which $z\in pq \subseteq w_\alpha w_\beta,\, z'\in w_\alpha w_\beta$\,.}

The second case is when $z$ still belongs to an internal segment $pq$ contained in the side $w_\alpha w_\beta\subseteq \Q'$, and also $z'$ belongs to the side $w_\alpha w_\beta$. In particular, if also $z'$ is in the internal segment $pq$ then we already know the validity of~(\ref{bilpr}) because $\uad'(z)=\upqd(z)$ and $\uad'(z')=\upqd(z')$, while $\upqd$ is $4L$ bi-Lipschitz. Therefore, we can directly assume that $z' \in w_\alpha p$, being the case $z' \in q w_\beta$ clearly the same.\par

By Definition~\ref{eq:valpha} we know that $\uad'(z')=\uad(z')$ lies in the segment $u(w_\alpha) u(p)$, which is a radius of the ball $\B\big(u(w_\alpha),\xi_\alpha\big)$. Hence, for any point $s$ outside the same ball, a trivial geometric argument tells us that
\begin{equation}\label{uffi}
\big|s - \uad(z')\big| \geq \frac{\big|s - u(p)\big| + \big|u(p) - \uad(z')\big|}3\,.
\end{equation}
Notice now that it is not true, in general, that $\uad'(z)=\upqd(z)$ lies outside the ball $\B\big(u(w_\alpha),\xi_\alpha\big)$. However, recalling that $\upqd$ is an interpolation of $u$, by Definition~\ref{def:interp} we know that $\upqd(z)$ is in a segment whose both extremes are out of the ball. Thus, if $\big|\upqd(z)-u(w_\alpha)\big| <\xi_\alpha$, up to possibly decreasing $\delta(\alpha,\,\beta)$ it is surely true that
\[
\xi_\alpha - \big|\upqd(z)-u(w_\alpha)\big| \ll \big|\upqd(z) - u(p)\big| \,.
\]
Putting this observation together with~(\ref{uffi}) we readily obtain that
\[\begin{split}
\big|\upqd(z) - \uad(z')\big| &\geq \frac{\big|\upqd(z) - u(p)\big| + \big|u(p) - \uad(z')\big|}4\\
&= \frac{\big|\upqd(z) - \upqd(p)\big| + \big|\uad(p) - \uad(z')\big|}4\,,
\end{split}\]
recalling that $\uad(p)=\upqd(p)=u(p)$ (of course, by selecting $\delta(\alpha,\beta)$ small enough, we could have used any number greater than $3$, instead of $4$, in the above estimate). Therefore, since $\upqd$ is $4L$ bi-Lipschitz while $\uad$ is $18L$ bi-Lipschitz, we readily obtain
\[\begin{split}
\big|\uad'(z)-\uad'(z')\big| &= \big| \upqd(z)  - \uad(z')\big|
\geq \frac{\big|\upqd(z) - \upqd(p)\big|}4 + \frac{\big|\uad(p) - \uad(z')\big|}4\\
&\geq \frac{|z -p\big|}{16L} + \frac{|p -z'|}{72L}
\geq \frac{|z -z'|}{72L}\,,
\end{split}\]
recalling that $z,\,p$ and $z'$ are aligned. Hence, (\ref{bilpr}) is checked once again also in this case.

\step{IV}{The case in which neither $z$ nor $z'$ are in some internal segment.}

Thanks to Step~II and Step~III and by the symmetry of the inequality~(\ref{bilpr}), we are left to consider only the situation where no one between $z$ and $z'$ is inside some internal segment. In other words, both $z$ and $z'$ must be either in $\Q\setminus\Q'$ or in some cross. By the definition of $\uad'$, this means that $\uad'(z)=\uad(z)$ and $\uad'(z')=\uad(z')$. And thus, since $\uad$ is $18L$ bi-Lipschitz thanks to Lemma~\ref{lem:cross}, the validity of~(\ref{bilpr}) is already known. Summarizing, we have shown the validity of~(\ref{bilpr}) in any possible case, and this means that the function $\uad':\Q\longrightarrow\Delta$ is injective and $72L$ bi-Lipschitz.

\step{V}{Conclusion.}

We have now to define the piecewise affine and bi-Lipschitz function $\tilde v_\eps:\Omega\setminus\Omega_\eps\longrightarrow \Delta\setminus\Delta_\eps$, matching $u$ on $\partial\Omega$ and matching $v_\eps$ on $\partial\Omega_\eps$. To do so, consider each square $\D_j$ of the tiling contained in $\Omega\setminus\Omega_\eps$. The function $\uad'$ is $72L$ bi-Lipschitz from $\partial\D_j$ to a subset of $\Delta$, then by Theorem~\ref{teo:bilsquare} it can be continuously extended to a piecewise affine bi-Lipschitz function of the whole square $\D_j$, with bi-Lipschitz constant $C_3 72^4 L^4$. Define $\tilde v_\eps$ as the countably piecewise affine function on $\Omega\setminus\Omega_\eps$ which gathers all these extensions on all the squares $\D_j\subseteq \Omega\setminus\Omega_\eps$ of the tiling.\par
For each square $\D_j\subseteq\Omega\setminus\Omega_\eps$, we clearly have that $\partial\big(\tilde v_\eps(\D_j)\big)= \uad'(\partial\D_j)$. This yields that $\tilde v_\eps$ is injective. Moreover, by continuity it is clear that $\tilde v_\eps =u$ on $\partial \Omega$, and by construction $\tilde v_\eps = v_\eps$ on $\partial \Omega_\eps$. As a consequence, $\tilde v_\eps:\Omega\setminus\Omega_\eps\longrightarrow \Delta\setminus\Delta_\eps$ fulfills all our requirements. In particular, one has $C_1=72^4 C_3$.
\end{proof}

Let us now make a simple observation, which will be useful in the sequel.

\begin{remark}\label{remfin}
Assume that $\Omega$ is a right polygon of side-length $\bar r$ and that $u$ is $\bar r$-piecewise affine on $\partial\Omega$, according to Definition~\ref{lastdef}. Then, consider the $r$-polygon $\Omega_\eps$ given by Proposition~\ref{prop:lebappr}. By the construction of Section~\ref{sect:1}, it is not restrictive to assume that $\bar r \in r\N$, and that $\Omega_\eps$ is a subset of the $r$-tiling of $\Omega$. Therefore, we can repeat \emph{verbatim} the construction of Proposition~\ref{prop:lipext} using, as tiling, the $r$-tiling of $\Omega$. Notice that in this case assumption~(\ref{second}) is not valid --see the remark right after~(\ref{second})-- but in fact if $\Omega$ is an $\bar r$-polygon, and $u$ is $\bar r$-piecewise affine on $\partial\Omega$, there is no need for the tiling to use smaller and smaller squares at the boundary. As a consequence, the bi-Lipschitz approximation $\tilde v_\eps$ provided by Proposition~\ref{prop:lipext} is (finitely) piecewise affine instead of countably piecewise affine. Observe that the assumption that $u$ is $\bar r$-piecewise affine on $\partial\Omega$ is essential, because otherwise the map $\tilde v_\eps$ would not coincide with $u$ on $\partial\Omega$.
\end{remark}

To conclude the proof of Theorem~\ref{main}, we then only need to give the proofs of Lemma~\ref{lem:curves} and of Lemma~\ref{lem:cross}.

\begin{proof}[Proof of Lemma~\ref{lem:curves}: ]
Let $\rho>0$ be a small number, to be fixed later. Define then $t_0=0$, $z_0=p$ and then recursively
\begin{align*}
t_{i+1} := \max \Big\{ 1\geq t>t_i:\, \big| u(z_i)- u\big( p + t(q-p)\big)\big|\leq \rho\Big\}\,, &&
z_{i+1} := p + t_{i+1}(q-p)\,.
\end{align*}
In this way, we have selected a finite sequence of points $z_0= p,\, z_1,\, \dots z_N=q$ in the segment $pq$, where $N=N(p,q,\rho)$. We can then already define the function $\upqd$ by setting, for any $0\leq i \leq N-1$ and any $0\leq t \leq 1$,
\[
\upqd\Big(z_{i}+t\big(z_{i+1}-z_i\big)\Big)=u(z_{i})+t\big(u(z_{i+1})-u(z_i)\big)\,.
\]
Hence, $\upqd$ is the interpolation of $u$ associated with the points $\{z_i\}$, according with Definition~\ref{def:interp}. The function $\upqd$ is by construction finitely piecewise affine and $L$-Lipschitz. By the uniform continuity of $u$ in $pq$ it is also clear that the bound $\|u-\upqd\|_{L^\infty(pq)}\leq \delta$ holds true as soon as $\rho$ is small enough. To conclude, we have thus only to check that
\begin{equation}\label{caro21}
\big|\upqd(z) - \upqd(z') \big| \geq \frac{1}{4L} |z-z'|
\end{equation}
for all $z,\, z'$ in $pq$. If both $z$ and $z'$ belong to the same segment $z_i z_{i+1}$, then the estimate is true because $\upqd$ is affine on that segment and $u$ is $L$ bi-Lipschitz.\par

Assume then that $z\in z_i z_{i+1}$ and $z'\in z_j z_{j+1}$ with $j>i$. If $j=i+1$, that is, $z$ and $z'$ belong to two consecutive segments, then by the definition of the points $z_i$ the angle $\angle{\upqd(z)}{\upqd(z_{i+1})}{\upqd(z')}$ is at least $\pi/3$. Hence,
\[\begin{split}
\big|\upqd(z) - \upqd(z') \big| &\geq \frac{\big|\upqd(z) - \upqd(z_{i+1}) \big|}2  + \frac{\big|\upqd(z_{i+1}) - \upqd(z') \big|}{2}\\
&=|z-z_{i+1}|\,\frac{\big|u(z_i) -u(z_{i+1})|}{2|z_i - z_{i+1}|}+|z_{i+1}-z'|\,\frac{\big|u(z_{i+1}) -u(z_{i+2})|}{2|z_{i+1} - z_{i+2}|}
\geq \frac{|z-z'|}{2L}\,,
\end{split}\]
so that~(\ref{caro21}) is checked.\par
Instead, let us see what happens if $j>i+1$. In this case, since $\upqd(z)\in u(z_i) u(z_{i+1})$ and for all $l>i+1$ one has $u(z_l)\notin \B(u(z_i),\rho)\cup \B(u(z_{i+1}),\rho)$, an immediate geometric argument ensures that $|\upqd(z)-\upqd(z')|\geq \sqrt{3}\rho/2$. As a consequence, we have
\[
\big|u(z_i) - u(z_{j+1}) \big| \leq \big|\upqd(z) -\upqd(z') \big| + 2 \rho
\leq \bigg(1+\frac 43 \,\sqrt{3} \bigg) \big|\upqd(z) -\upqd(z') \big|
\leq 4 \big|\upqd(z) -\upqd(z') \big|\,,
\]
which yields
\[
\big|\upqd(z) -\upqd(z') \big| \geq \frac{\big|u(z_i) - u(z_{j+1}) \big|}4
\geq \frac{\big|z_i - z_{j+1}\big|}{4L}
\geq \frac{\big|z - z'\big|}{4L}\,,
\]
hence~(\ref{caro21}) holds true also in this case and we conclude the proof.
\end{proof}

\begin{proof}[Proof of Lemma~\ref{lem:cross}:  ]
Let us take a vertex $w_\alpha$ of the grid $\Q'$. Take then a constant $\xi_\alpha\leq \ell_\alpha/(3L)$, with $\xi_\alpha = \ell_\alpha/(3L)=r/(3L)$ if $w_\alpha\in\partial\Omega_\eps$, while if $w_\alpha\notin\partial\Omega_\eps$ the inequality can be strict. In particular, it is admissible to ask that for any $\alpha$ one has
\begin{equation}\label{inpa}
\xi_\alpha < \frac{r}{2L}\,.
\end{equation}
Define now $\xi_\alpha^i$ as in Definition~\ref{eq:valpha} and, for any $1\leq i \leq \bar i(\alpha)$, let $p_i=w_\alpha + \xi_\alpha^i\big(w_\alpha^i-w_\alpha\big)$. If $w_\alpha\in\Omega\setminus\partial\Omega_\eps$, then we have
\begin{equation}\label{trfa}
u(w_\alpha) u(p_i)\subset\subset\Delta \qquad \forall\, 1\leq i\leq \bar i(\alpha)\,,
\end{equation}
up to possibly decrease the value of $\xi_\alpha$. Instead, if $w_\alpha\in\partial\Omega_\eps$, then~(\ref{trfa}) is already ensured by~(\ref{luiro}) and~(\ref{luiue}) in Proposition~\ref{prop:lebappr}, without any need of changing $\xi_\alpha$.\par
We introduce then the adjusted function $\uad$ of Definition~\ref{eq:valpha}: to obtain the thesis, we need to check that it fulfills the requirements of Lemma~\ref{lem:cross}. Thanks to~(\ref{trfa}), we already know that $\uad:\Q\longrightarrow \Delta$. Hence, all we have to do is to check that
\begin{equation}\label{awhtd}
\frac{|z-z'|}{18L} \leq \big|\uad(z) - \uad(z')\big| \leq 18L |z-z'|\,.
\end{equation}
for all $z,\, z' \in \Q$. We will do it in some steps.

\step{I}{For all $\alpha$, $\uad^{-1}\Big(\clos\B\big(u(w_\alpha),\xi_\alpha\big)\Big)=Z_\alpha$.}

We start observing an important property, that is, for any $\alpha$ and for any $z\in\Q$ we have that $\big|\uad(z)-u(w_\alpha)\big|\leq \xi_\alpha$ if and only if $z\in Z_\alpha$. In fact, if $z\in Z_\alpha$ then $z\in w_\alpha p_i$ for some $1\leq i \leq \bar i(\alpha)$, and since $\uad$ is affine in the segment $w_\alpha p_i$, while $\big|\uad(p_i)-u(w_\alpha)\big|=\xi_\alpha$, then of course $\big|\uad(z)-u(w_\alpha)\big|\leq \xi_\alpha$.\par
On the other hand, assume that $z\notin Z_\alpha$: we have to show that $\big|\uad(z)-u(w_\alpha)\big|>\xi_\alpha$. If $z\in w_\alpha w_\alpha^i$ for some $1\leq i \leq \bar i(\alpha)$, then there are three possibilities. First, if $w_\alpha w_\alpha^i\subset\Q\setminus\Q'$, then $\uad=v_\eps$ is affine on the side $w_\alpha w_\alpha^i$, so the claim is trivial. Second, if $w_\alpha w_\alpha^i \subset \Q'$ and $z$ belongs to the cross $Z_\beta$ associated to the vertex $w_\beta=w_\alpha^i$, then again the claim is immediate since $\uad(z)$ belongs to the ball $\B\big(u(w_\beta),\xi_\beta\big)$, which does not intersect $\B\big(u(w_\alpha),\xi_\alpha\big)$ by Remark~\ref{rightafter}. Lastly, if $w_\alpha w_\alpha^i\subset\Q'$ and $z\notin Z_\beta$, then $\uad(z)=u(z)$, thus the claim is again obvious by the definition of $\xi_\alpha^i$.\par
To conclude the step, we have to consider a point $z\notin Z_\alpha$ which does not belong to any side of $\Q$ starting at $w_\alpha$. We have again to distinguish some possible cases. If $z$ belongs to the cross $Z_\beta$ for some $\beta$, then again the claim follows by the fact that $\B\big(u(w_\alpha),\xi_\alpha\big) \cap \B\big(u(w_\beta),\xi_\beta\big)=\emptyset$. If $z$ does not belong to any cross and $z\in \Q'$, then $\uad(z)=u(z)$ so the claim follows because, using the bi-Lipschitz property of $u$ and the fact that $\xi_\alpha\leq \ell_\alpha/(3L)$, we have
\[
u(z) \in \B\big( u(w_\alpha), \xi_\alpha\big)\quad \Longrightarrow \quad \big|z - w_\alpha\big| \leq \frac{\ell_\alpha}{3}\,,
\]
which is impossible because $|z-w_\alpha|> \ell_\alpha$. Finally, consider the case when $z\in \Q\setminus \Q'$. In this case, we surely have $|z-w_\alpha|\geq r$ by construction, thus by~(\ref{luiro}) and~(\ref{inpa}) we get
\[\begin{split}
\big|\uad(z)-u(w_\alpha)\big| &= \big|v_\eps(z)-u(w_\alpha)\big|
\geq \big|u(z)-u(w_\alpha)\big| - \big|u(z) - v_\eps(z)\big|\\
&\geq \frac{\big| z - w_\alpha\big|}L - \frac{\sqrt{2} r}{6L^3}
\geq  \frac{r}{2L}
> \xi_\alpha\,,
\end{split}\]
thus the first step is concluded.\par

Now, taken two points $z,\, z'\in\Q$, we have to show the validity of~(\ref{awhtd}).

\step{II}{Validity of~(\ref{awhtd}) if $z,\,z'\in Z_\alpha$.}

Let us first suppose that both $z$ and $z'$ belong to the same cross $Z_\alpha$. By construction, $\uad$ is $L$ bi-Lipschitz on each segment $w_\alpha p_i$, hence to show~(\ref{awhtd}) we can assume without loss of generality that $z\in w_\alpha p_1$ and $z'\in w_\alpha p_2$. Therefore, on one side we have
\[\begin{split}
\big|\uad(z) - \uad(z')\big| &\leq \big|\uad(z) - \uad(w_\alpha)\big| + \big|\uad(w_\alpha) - \uad(z')\big|
\leq L \big( |z-w_\alpha| + |w_\alpha - z'|\big)\\
&\leq \sqrt{2} \,L\, |z-z'|\,.
\end{split}\]
On the other side, to estimate $\big|\uad(z) - \uad(z')\big|$ from below, assume without loss of generality that $\big|\uad(w_\alpha)-\uad(z)\big|\leq \big|\uad(w_\alpha)-\uad(z')\big|$, and define $z''\in w_\alpha z'$ so that
\[
\big|\uad(w_\alpha)-\uad(z)\big|=\big|\uad(w_\alpha)-\uad(z'')\big|\,.
\]
Since the triangle $\uad(w_\alpha)\uad(z)\uad(z'')$ is isosceles, then
\begin{equation}\label{ban}
\angle{\uad(z)}{\uad(z'')}{\uad(z')}\geq \frac \pi 2\,.
\end{equation}
Moreover, we claim that
\begin{equation}\label{psp}
\frac{\big|\uad(z) - \uad(z'')\big|}{\big| z - z''\big|}\geq \frac 1 {2L}\,.
\end{equation}
Indeed, if both $w_\alpha w_\alpha^1$ and $w_\alpha w_\alpha^2$ belong to $\Q'$, then by definition
\[
\frac{\big|\uad(z) - \uad(z'')\big|}{\big| z - z''\big|} = \frac{\big|\uad(p_1) - \uad(p_2)\big|}{\big| p_1 - p_2\big|}
=\frac{\big|u(p_1) - u(p_2)\big|}{\big| p_1 - p_2\big|}\geq \frac 1 L\,,
\]
so~(\ref{psp}) holds true. Conversely, if both $w_\alpha w_\alpha^1$ and $w_\alpha w_\alpha^2$ belong to $\Q\setminus\Q'$, then since $v_\eps$ is $L+\eps$ bi-Lipschitz we have
\[
\frac{\big|\uad(z) - \uad(z'')\big|}{\big| z - z''\big|} = \frac{\big|\uad(p_1) - \uad(p_2)\big|}{\big| p_1 - p_2\big|}
=\frac{\big|v_\eps(p_1) - v_\eps(p_2)\big|}{\big| p_1 - p_2\big|}\geq \frac 1 {L+\eps}\,,
\]
so again~(\ref{psp}) holds true. Finally, assume that $w_\alpha w_\alpha^1\subseteq \Q'$ while $w_\alpha w_\alpha^2\subseteq \Q\setminus \Q'$ (the case of $w_\alpha w_\alpha^1\subseteq \Q\setminus\Q'$ and $w_\alpha w_\alpha^2\subseteq \Q'$ being completely equivalent). In this case, it must clearly be $w_\alpha\in\partial\Omega_\eps$, hence by Remark~\ref{rightafter} we know that $|p_1-w_\alpha|$ and $|p_2-w_\alpha|$ are both at least $\xi_\alpha/L=r/(3L^2)$, thus $|p_1-p_2|\geq \sqrt{2} r/(3L^2)$. Therefore, recalling again~(\ref{luiro}), we have
\[\begin{split}
\frac{\big|\uad(z) - \uad(z'')\big|}{\big| z - z''\big|} &= \frac{\big|\uad(p_1) - \uad(p_2)\big|}{\big| p_1 - p_2\big|}
=\frac{\big|u(p_1) - v_\eps(p_2)\big|}{\big| p_1 - p_2\big|}\\
&\geq \frac{\big|u(p_1) - u(p_2)\big|}{\big| p_1 - p_2\big|} - \frac{\big|u(p_2) - v_\eps(p_2)\big|}{\big| p_1 - p_2\big|}
\geq \frac{1}{L}- \frac{\sqrt{2} r / (6L^3)}{\sqrt{2} r/(3L^2)}
=\frac 1{2L}\,,
\end{split}\]
thus~(\ref{psp}) has been finally checked in all the possible cases. This inequality, together with~(\ref{ban}) and again with the fact that $\uad$ is $L$ bi-Lipschitz on the segment $z' z'' \subseteq w_\alpha p_2$, yields
\[\begin{split}
\big|\uad(z) - \uad(z')\big| &\geq \frac{\sqrt{2}} 2\Big(\big|\uad(z) - \uad(z'')\big| + \big|\uad(z'') - \uad(z')\big|\Big)\\
&\geq \frac{\sqrt{2}}{2}\bigg( \frac{|z-z''|}{2L} + \frac{|z'' - z'|}{L}\bigg)
\geq \frac{\sqrt{2}}{4L} |z-z'|\,.
\end{split}\]
Summarizing, under the assumptions of this step
\begin{equation}\label{step2}
\frac{\sqrt{2}}{4L} |z-z'| \leq \big|\uad(z) - \uad(z')\big| \leq \sqrt{2} L |z-z'|\,.
\end{equation}
Therefore, (\ref{awhtd}) is shown and this step is concluded.

\step{III}{Validity of~(\ref{awhtd}) if for all $\alpha$ one has $z,\,z'\notin \mathrm{int} \, Z_\alpha$.}

Consider now the situation when neither $z$ nor $z'$ belong to the interior of any cross. In this case, we have that $\uad(z)=u(z)$ if $z\in \Q'$, while $\uad(z)=v_\eps(z)$ if $z\in\Q\setminus\Q'$, and the same holds for $z'$. Since $u$ is $L$ bi-Lipschitz while $v_\eps$ is $L+\eps$ bi-Lipschitz, the validity of~(\ref{awhtd}) is obvious if both $z,\,z'\in \Q'$, as well as if both $z,\,z'\in\Q\setminus\Q'$. Therefore, we can just concentrate on the case in which $z\in \Q',\, z'\in\Q\setminus\Q'$.\par
In this case, the main observation is that $|z-z'| \geq \sqrt{2} r /(3L^2)$, since both $z$ and $z'$ must be at distance at least $r/(3L^2)$ from any vertex $w_\alpha\in\partial\Omega_\eps$, because they do not belong to any cross $Z_\alpha$. As a consequence, again by~(\ref{luiro}) we get
\[\begin{split}
\big|\uad(z) - \uad(z')\big| &= \big|u(z) - v_\eps(z')\big|
\geq \big|u(z) - u(z')\big| - \big|u(z') - v_\eps(z')\big|\\
&\geq \frac{|z-z'|}{L} - \frac{\sqrt{2}r}{6L^3}
\geq \frac{|z-z'|}{2L}\,,
\end{split}\]
while
\[\begin{split}
\big|\uad(z) - \uad(z')\big| &= \big|u(z) - v_\eps(z')\big|
\leq \big|u(z) - u(z')\big| + \big|u(z') - v_\eps(z')\big|\\
&\leq L \, |z-z'|+\frac{\sqrt{2}r}{6L^3}
\leq \bigg( L+ \frac{1}{2L} \bigg)|z-z'|\,,
\end{split}\]
thus also in this case~(\ref{awhtd}) is proven (keep in mind that, since $u$ is a $L$ bi-Lipschitz map, then of course $L\geq 1$!). In particular, under the assumptions of this step one has
\begin{equation}\label{step3}
\frac{|z-z'|}{2L}\leq \big|\uad(z) - \uad(z')\big| \leq  \frac 32\, L\,|z-z'|\,.
\end{equation}

\step{IV}{Validity of~(\ref{awhtd}) if $z\in  Z_\alpha$ and for all $\beta$ one has $z'\notin \mathrm{int}\, Z_\beta$.}

We pass now to consider the case when $z$ belongs to some cross $Z_\alpha$, while $z'$ does not belong to the interior of any cross. In particular, we can assume that $z\in w_\alpha p_1$. To get the above estimate in~(\ref{awhtd}), it is enough to make a trivial geometric observation, namely, that there exists $1\leq i \leq \bar i(\alpha)$ such that
\[
|z-z'| \geq \frac{\sqrt{2}}2\Big( |z-p_i|+|p_i-z'|\Big)\,,
\]
not necessarily with $i=1$. As a consequence, we can use the estimate~(\ref{step2}) of Step~II for the points $z$ and $p_i$ --which both belong to $Z_\alpha$-- and the estimate~(\ref{step3}) of Step~III for the points $p_i$ and $z'$ --none of which belongs to the interior of some $Z_\beta$-- to get
\[\begin{split}
\big|\uad(z) - \uad(z')\big| &\leq \big|\uad(z) - \uad(p_i)\big| + \big|\uad(p_i) - \uad(z')\big|
\leq \sqrt{2} L |z-p_i| + \frac 32\, L |p_i - z'|\\
&\leq \frac 32\, \sqrt{2} \,L |z-z'|\,.
\end{split}\]
On the other hand, to get the below estimate in~(\ref{awhtd}), let us recall that by Step~I we have
\begin{align}\label{uffi2}
\uad(z) \in \clos \B\big(u(w_\alpha),\xi_\alpha\big)\,, &&
\uad(z') \notin \B\big(u(w_\alpha),\xi_\alpha\big)\,.
\end{align}
Since $\uad(z)$ belongs to the radius $u(w_\alpha)\uad(p_1)$, then an immediate geometric argument from~(\ref{uffi2}) implies, as already observed in~(\ref{uffi}), that
\begin{equation}\label{alla}
\big|\uad(z) - \uad(z')\big| \geq \frac{\big|\uad(z) - \uad(p_1)\big| + \big|\uad(p_1) - \uad(z')\big|}3\,.
\end{equation}
Thus, using the $L$ bi-Lipschitz property of $\uad$ in the segment $w_\alpha p_1$, and the estimate~(\ref{step3}) of Step~III for $p_1$ and $z'$, we get
\[
\big|\uad(z) - \uad(z')\big| \geq \frac{|z-p_1|}{3L} + \frac{|p_1-z'|}{6L} \geq \frac{|z-z'|}{6L}\,.
\]
Summarizing, under the assumptions of this step we have
\begin{equation}\label{step4}
\frac{|z-z'|}{6L} \leq \big|\uad(z) - \uad(z')\big| \leq \frac 32\, \sqrt{2} \,L |z-z'|\,,
\end{equation}
hence in particular~(\ref{awhtd}) is again checked.

\step{V}{Validity of~(\ref{awhtd}) if $z\in Z_\alpha$ and $z'\in Z_\beta$.}

The last situation which is left to consider is when $z$ and $z'$ belong to two different crosses. This situation will be very similar to that of Step~IV. Indeed, for the above estimate in~(\ref{awhtd}) we can again start observing that for some $1\leq i \leq \bar i (\alpha)$ it must be
\[
|z-z'| \geq \frac{\sqrt{2}}2\Big( |z-p_i|+|p_i-z'|\Big)\,.
\]
Then, we use the estimate~(\ref{step2}) of Step~II for the points $z,\, p_i\in Z_\alpha$, and the estimate~(\ref{step4}) of Step~IV for the points $z'\in Z_\beta$ and $p_i$ --which does not belong to the interior of any cross-- getting
\[\begin{split}
\big|\uad(z) - \uad(z')\big| &\leq \big|\uad(z) - \uad(p_i)\big| + \big|\uad(p_i) - \uad(z')\big|\\
&\leq \sqrt{2} L |z-p_i| + \frac 32\,\sqrt{2}\, L |p_i - z'|
\leq 3L |z-z'|\,.
\end{split}\]
Finally, to find the below estimate in~(\ref{awhtd}) we notice again that~(\ref{alla}) holds true, and we use the $L$ bi-Lipschitz property of $\uad$ in $w_\alpha p_1$ and the estimate~(\ref{step4}) of Step~IV for $p_1$ and $z'$, obtaining
\[
\big|\uad(z) - \uad(z')\big| \geq \frac{|z-p_1|}{3L} + \frac{|p_1-z'|}{18L} \geq \frac{|z-z'|}{18L}\,.
\]
Thus, we have finally checked~(\ref{awhtd}) in all the possible cases and the proof is concluded.
\end{proof}

\section{Finitely piecewise affine approximation on polygonal domains\label{sect:3}}

In this last section we give a proof of Theorem~\ref{mainaffine}. In fact, the proof is quite short, since it is just a simple adaptation of the arguments of Section~\ref{sect:2}.

\begin{proof}[Proof of Theorem~\ref{mainaffine}: ]

First of all, assume that $\Omega$ is an $\bar r$-right polygon and that $u$ is $\bar r$-piecewise affine on $\partial \Omega$ according to Definition~\ref{lastdef}. Then, as already underlined in Remark~\ref{remfin}, we can slightly modify the proofs of Proposition~\ref{prop:lebappr} and Proposition~\ref{prop:lipext} to get what follows. First of all, there exist some $r$ such that $\bar r \in r\N$, an $r$-right polygon $\Omega_\eps\subset\subset\Omega$ which is part of the $r$-tiling of $\Omega$, and an $L+\eps$ bi-Lipschitz and piecewise affine function $v_\eps:\Omega_\eps\longrightarrow \R^2$ for which~(\ref{luino}), (\ref{eq:lebappr3}), (\ref{luiue}) and~(\ref{luiro}) hold. Moreover, there exists also a finitely piecewise affine map $\tilde v_\eps:\Omega\setminus\Omega_\eps \longrightarrow \Delta\setminus\Delta_\eps$ which is $C_1L^4$ bi-Lipschitz and which coincides with $u$ on $\partial\Omega$ and with $v_\eps$ on $\partial\Omega_\eps$. Therefore, gluing $v_\eps$ and $\tilde v_\eps$ exactly as in the proof of Theorem~\ref{main}, we immediately get the required $C_1L^4$ bi-Lipschitz and (finitely) piecewise affine approximation of $u$.\par

Consider now the general situation of a polygon $\Omega$ with a map $u$ which is piecewise affine on $\partial\Omega$. Of course, there exist a right polygon $\widehat \Omega$ and a (finitely) piecewise affine and bi-Lipschitz map $\Phi:\Omega \longrightarrow \widehat\Omega$, having bi-Lipschitz constant $C=C(\Omega)$. The map $u\circ \Phi^{-1}$ is a $C L$ bi-Lipschitz map from the right polygon $\widehat \Omega$ to $\Delta$, which is piecewise affine on the boundary. Then, we can apply the first part of the proof to get an approximation $v:\widehat \Omega\longrightarrow\Delta$ which is finitely piecewise affine and $C_1 C^4 L^4$ bi-Lipschitz. Finally, $v\circ \Phi:\Omega\longrightarrow \Delta$ is a $C_1 C^5 L^4$ bi-Lipschitz approximation of $u$ as desired. Thus, the proof is concluded by setting $C'(\Omega)=C^5$.
\end{proof}

\begin{remark}\label{depC'}
Observe that the (best) constant $C'(\Omega)$ depends on the geometric features of $\Omega$, such as the minimum and the maximum angles of its boundary. However, by the construction above one has that $C'(\Omega)=1$ whenever $\Omega$ is a right polygon.
\end{remark}

\section*{Acknowledgments}

Both authors have been supported by the ERC Starting Grant n. 258685 ``Analysis of Optimal Sets and Optimal Constants: Old Questions and New Results'', and by the ERC Advanced Grant n. 226234 ``Analytic Techniques for Geometric and Functional Inequalities''.

\end{document}